\documentclass{article}

\usepackage{enumerate}
\usepackage{amsthm}
\usepackage{mathrsfs}
\usepackage{amsmath}
\usepackage{commath}
\usepackage{amsfonts}
\usepackage{color}
\usepackage[hmargin=1.25in,vmargin=1in]{geometry}
\usepackage{tikz}
\usepackage{pgfplots}
\pgfplotsset{compat=1.18} 
\usepackage[backend=bibtex, style=alphabetic]{biblatex}

\newtheorem{thm}{Theorem}
\newtheorem{dfn}[thm]{Definition}
\newtheorem{prop}[thm]{Proposition}
\newtheorem{cor}[thm]{Corollary}
\newtheorem{rem}[thm]{Remark}
\newtheorem{lem}[thm]{Lemma}
\newtheorem{expl}[thm]{Example}

\newtheorem{clm}[thm]{Claim}
\newtheorem{cons}[thm]{Construction}

\numberwithin{thm}{section}
	
\newcommand{\pf}{\begin{proof}}
\newcommand{\epf}{\end{proof}}
\newcommand\be{\begin{equation}}
\newcommand\ee{\end{equation}}
\newcommand\bnu{\begin{enumerate}}
\newcommand\enu{\end{enumerate}}

\newcommand\Nb{\mathbb{N}}

\newcommand\Rb{\mathbb{R}}
\newcommand\Sb{\mathbb{S}}

\newcommand\Hc{\mathcal{H}}
\newcommand{\Lc}{\mathcal{L}}

\newcommand{\id}{\mathrm{id}}

\newcommand{\diam}[1]{\operatorname{diam}(#1)}

\newcommand\dvert[2]{\left.\frac{\drm}{\drm #1}\right\rvert_{#1 = #2}}
\newcommand\drm{\mathrm{\,d}}

\newcommand{\argmin}[1]{\operatorname*{argmin}\limits _{#1}}

\title{Reifenberg Theorem for Locally Finitely Almost Splitting Sets}
\author{Jiaqi Zang}

\begin{document}

\maketitle

\begin{abstract}
    The well-known Reifenberg theorem states that if a subset of $\Rb^n$ can be well approximated by $k$-planes at every point and every scale, then it is biH\"older homeomorphic to a $k$-disk. 
    This article concerns a subset $S$ of $\Rb^n$ which can be approximated by at most $N$ parallel $k$ planes at each point and scale.  As a subset of $\Rb^n$ such an $S$ may be quite degenerate; $S$ may clearly not be homeomorphic to a disk, and indeed we will see may not be homeomorphic to a union of disks.  However, we prove that $S$ is still the image of a {\it multivalued} map on $\Rb^k$, which is itself a biH\"older homeomorphism of the disk into the set of subsets of $\Rb^n$.
\end{abstract}

\tableofcontents

\section{Introduction}

The Reifenberg theory concerns a set $S$ (or a measure $\mu$) that bears a local rigidity structure. To be precise, it assumes that $S$ is locally approximated well by some well-behaved object at each point and at every scale in some sense. Depending on the setting, for each neighborhood, the local approximator can be a plane\cite{Rei60}, a plane with holes\cite{DT09}, a measure concentrated on a plane\cite{NV17}, the zero of a harmonic map \cite{BET17} or other objects. The local approximation can happen in the Hausdorff sense\cite{Rei60}, Gromov Hausdorff sense\cite{CC97}, cumulative $L^2$ sense\cite{NV17} or other reasonable relatively weak senses. The conclusions are that $S$ will have a nice global structure that resembles the local approximators. In philosophy, the local ridigity can pass to a global topological or geometric structure.

Our result in this article is a generalization of the classical Reifenberg theorem. As the starting point of this article, we should first introduce the classical Reifenberg theorem. In the classical Reifenberg, a subset $S\subset \Rb^n$ is assumed to be Hausdorff close to a $k$ affine plane at all points and for all scales. In this case, Reifenberg proves that $S$ is a biH\"older image of a ball in $\Rb^k$.

Recall the Hausdorff distance in $\Rb^n$, which is an $L^\infty$ distance for subsets.

\begin{dfn}

Let $A,B\subset \Rb^n$. Define their Hausdorff distance to be

\begin{equation}
d_H(A,B):=\inf\{s>0:A\subset B_s(B),B\subset B_s(A)\}.
\end{equation}
    
\end{dfn}

We also use a local version of Hausdorff distance in this article.
\begin{dfn}[local Hausdorff distance]

For $A,B\subset\Rb^n$, $B_r(x)\subset \Rb^n$. \begin{equation}d_H|_{B_r(x)}(A,B):=\inf\{s>0:A\cap B_r(x)\subset B_s(B),B\cap B_r(x)\subset B_s(A)\}.\end{equation}
    
\end{dfn}

A set $S\subset \Rb^n$ satisfies the classical Reifenberg condition if it is locally Hausdorff close to a $k$-plane.

\begin{dfn}[Reifenberg condition]

Let $S\subset B_2(0^n)$ be a closed subset. For $k \le n, \delta >0$, we say $S$ satisfies the $(k,\delta)$-Reifenberg condition if for each $B_r(x)\subset B_2(0^n)$, there is a $k$-affine plane $l_{x,r}\subset \Rb^n$ with
\begin{equation}d_H|_{B_r(x)}(S,l_{x,r})\le \delta r.\end{equation}
    
\end{dfn}

Now let's state the classical Reifenberg theorem.

\begin{thm}[Classical Reifenberg Theorem]

Assume $S\subset B_2(0^n)$ satisfies the $(k,\delta)$-Reifenberg condition. Then for each $\alpha>0$, there exists $C(n)$ and $\delta(n,\alpha)>0$, such that if $\delta<\delta(n,\alpha)$, we have a $C^{1-\alpha}$-biH\"older $\phi: B_{1}(0^k)\to S$, i,e. for any $|x-y|\le 1$,
\begin{equation}
(1-C(n)\delta))|x-y|^{\frac{1}{1-\alpha}}\le|\phi(x)-\phi(y)|\le (1+C(n)\delta)|x-y|^{1-\alpha}.
\end{equation}
Moreover, $S\cap B_1(0^n)\subset \phi(B_{1}(0^k))$.
\end{thm}

In this note, we'll generalize the theorem to the multi-sheeted case. In this case, $S$ is locally Hausdorff close to multiple parallel $k$-planes, or a $k$-splitting set, instead of one $k$-plane in the classical case.

\begin{dfn}

A closed set $A\subset \Rb^n$ is said to be a $k$-splitting if there exists a $k$-dimensional linear subspace $\hat{l}$ with $\hat{l}+A=A$. $\hat{l}$ is called a splitting direction of $A$.

\end{dfn}

\begin{dfn}[$(k,\delta)$-splitting]

$S\subset \Rb^n$ is said to be a $(k,\delta)$-splitting in $B_r(x)\subset \Rb^n$ if there exists a $k$-splitting set $A$ s.t. 
\begin{equation}d_H|_{B_r(x)}(S,A)\le \delta r.
\end{equation}

\end{dfn}

In this article, we are not considering the most general case where the sheets could be arbitrarily dense. Indeed, if the splitting set is too dense then the splitting direction may not be well defined and other problems in the estimates may emerge as well. Instead, we only consider the case that the splitting set is always a finite collection of planes.

\begin{dfn}[splitting Reifenberg condition]

For a closed set $S\subset B_2(0^n)$, we say $S$ satisfies the $(k,\delta,N)$-splitting Reifenberg condition if the following holds.
\bnu

\item $d_H|_{B_2(0^n)}(S,\Rb^k\times\{0^{n-k}\})\le 2\delta.$
\item For each $B_r(x)\subset B_2(0^n)$, $S$ is a $(k,\delta)$-splitting with splitting set $A_{x,r}$, such that $|A_{x,r}^\perp|\le N$. Here $\hat{l}_{x,r}$ is the splitting direction and $A_{x,r}=\hat{l}_{x,r}+A_{x,r}^\perp$.

% For each $B_r(x)\subset B_2(0^n)$, there exists a $k$-splitting set $A_{x,r}$ with splitting direction $\hat{l}_{x,r}$ and $A_{x,r}=\hat{l}_{x,r} + A_{x,r}^\perp$, $|A_{x,r}^\perp|\le N$, such that
% \begin{equation}
%     d_H|_{B_r(x)}(S,A_{x,r})\le \delta r.
% \end{equation}
\enu
\end{dfn}
\begin{rem}
    Condition 1. is meant to avoid boundary issues, one can avoid them in other ways as well. See Remark \ref{main theorem general version}.
\end{rem}

It is clear that the $(k, \delta)$-Reifenberg condition is the same as $(k,\delta, 1)$-splitting Reifenberg condition. Thus the splitting Reifenberg condition is a generalization of the Reifenberg condition.

For each point and each scale, $S$ still has a well-defined splitting direction, but it is missing a unique approximating plane. Considering the classical Reifenberg theorem, one might guess that $S$ may be the union of a collection of H\"older manifolds, but $S$ can have a more subtle structure, as we will show by examples.

Recall that we may equip the space of subsets of $\Rb^n$ with Hausdorff distance.

\begin{dfn}

For $S\subset \Rb^n$, let $2^S$ be the space of subsets of $S$ and $2^S_N$ be the space of subsets of $S$ that contains at most $N$ points.
    
\end{dfn}

\begin{rem}
    We may equip both $2^S$ and $2^S_N$ with the Hausdorff distance $d_H$.
\end{rem}

\begin{expl}[merging lines]\label{merging lines}
\begin{figure}
    \centering
    \includegraphics[width=0.7\linewidth]{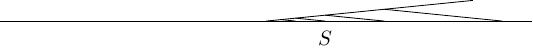}
    \caption{Example \ref{merging lines}}
    \label{fig1}
\end{figure}
(Figure \ref{fig1}) Let $f:\Rb\to 2^{\Rb}_3$, 
\begin{equation}
    f(x)=
\begin{cases}
    \{0\}, & x\le 0\\
    \{0,\delta(2^{i+1}-x),\delta x\},& 2^{i}<x\le 2^{i+1}\\
\end{cases}
\end{equation}
Take $S=Graph(f)\cap B_2(0^2)$. 
\end{expl}
\begin{proof}[Explanation]
In the above example, $f$ is $\delta$-Lipschitz as a map $\Rb\to 2^{\Rb}_3$. We have $d_H|_{B_2(0^n)}(S,\Rb\times\{0^{}\})\le 2\delta.$ For any $x=(x_1,x_2)$, $B_r(x)\subset B_2(0^2)$, we have that $S\cap B_r(x)$ is a $(1,\delta
)$-splitting with splitting set $A_{x,r}:= \Rb\times f(x_1)$. Thus $S$ satisfies the $(1,\delta,3)$-Reifenberg condition.

One should notice that $S$ cannot be written as a finite union of lines, since one line can only cover one segment from the upper line to the lower line. And there is no natural way to define the multiplicity on $S$.
    
\end{proof}

\begin{expl}[the twist]\label{the twist}
\begin{figure}
    \centering
    \includegraphics[]{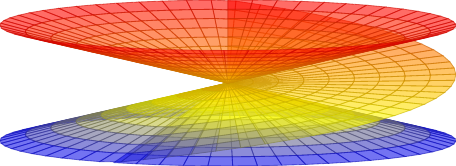}
    \caption{$S$ in Example \ref{the twist}} 
    \label{fig2}
\end{figure}
    (Figure \ref{fig2}) 
    For $\alpha\in(\pi,2\pi]$, define $f_\alpha:\Rb^2\to 2^{\Rb}_3$ as follows: for $x=(r\cos \theta, r\sin \theta)$, $\theta\in[0,2\pi)$, 
    \begin{equation}
        f(x)=\begin{cases}
            \{-r\delta,r\delta(\frac{2\theta}{\alpha}-1), r\delta\}&0\le\theta<\alpha,\\
            \{-r\delta,r\delta\}&\alpha\le\theta<2\pi.
        \end{cases}
    \end{equation}
    Take $S=Graph(f)\cap B_2(0^3)\subset \Rb^3$. 
\end{expl}

\pf[Explanation:] In the above example, $f$ is $4\delta$-Lipschitz as a map $\Rb^2\to (2^{\Rb}_3,d_H)$. We have $d_H|_{B_2(0^3)}(S,\Rb^2\times \{0\})\le 8\delta$. Take $B_r(x)\subset B_2(0^3)$, with $x=(x_1,x_2)\in \Rb^3$, $x_1\in \Rb^2$. Let $A_{x,r}=\Rb^2\times f(x_1)$. Then $d_H|_{B_{r_i}(x)}(S,A_{x,r})\le 4\delta r_i$. Thus $S$ satisfies the $(2,4\delta,3)$-splitting Reifenberg condition.

$S$ is the finite union of some disks, which wraps around the origin for $2+\alpha$ circles. In this case, one may possibly say the 'multiplicity' of $S$ at the origin is $2+\alpha$.
\epf

\begin{expl}\label{uncountable sheets}

Let $C\subset \Rb$ be the thin Cantor set defined as follows. Let $C_0 = [0,\delta]$, $C_2=[0,\delta^2]\cup[\delta-\delta^2,\delta
]$. For each $i$, $C_{i+1}$ is the left and right $\delta$ portion of each segment in $C_i$. Let $C:=\bigcap_iC_i$. Take $S=(C\times \Rb)\cap B_2(0^2)$.
    
\end{expl}

\pf[Explanation]

For each $x\in C$ and $r>0$, assume $\delta^{i-1}/4\ge r>\delta^{i}/4$. The interval $(x-r,x+r)$ intersects with at most one segment in $C_i$, and thus with at most two segments in $C_{i+1}$. $C\cap (x-r,x+r)$ can be covered by at most two intervals of length $\delta^{i+1}$, with $\frac{\delta^{i+1}/2}{r}<2\delta$. Therefore, for any $B_r(x)\subset B_2(0^2)$, $S$ is $(1,2\delta)$-splitting with a splitting set consisting of at most two lines, i.e. $S$ is a $(1,2\delta,2)$-splitting Reifenberg set.

$S$ is the union of lines. However, it is clear that the number of sheets of $S$ has the same cardinality as $\Rb$. Moreover, $S\cap \{0\}\times \Rb$ has a positive Hausdorff dimension.
\epf
From the above examples, we see that $S$ should not be seen as a union of sheets in general. Instead, we should understand $S$ as the image of some multivalued map. The main result of this article is the following theorem. It claims that $S$ can be viewed as a $k$-dimensional manifold $(1-\alpha)$-biH\"olderly embedded into the space $(2^{\Rb^n},d_H)$.

\begin{thm}\label{main theorem}

Given $\alpha>0$, $n, N\in \Nb_+$, assume $\delta<\delta(n,N,\alpha)$ is sufficiently small. Let $S\subset B_2(0^n)$ be a $(k,\delta,N)$-Reifenberg set. Then there exists a $C^{1-\alpha}$-biH\"older map $\iota:B_{1}(0^k)\to (2^{S},d_H)$, i.e., for each $c,d\in B_{1}(0^k)$ with $|c-d|\le 1$, we have $\iota(c),\iota(d)\subset S$, and
\begin{equation}
    (1-C(n,N)\delta)|c-d|^{\frac{1}{1-\alpha}}\le d_H(\iota(c),\iota(d))\le (1+C(n,N)\delta)|c-d|^{1-\alpha},
\end{equation}
such that 
\begin{equation}\label{S intersect B1 is covered}
S\cap B_1(0^n)\subset \iota(B_1(0^k)).
\end{equation}

\end{thm}
\begin{rem}
Since each $\iota(c)$ is a set, the above expression (\ref{S intersect B1 is covered}) should be interpreted as 
\begin{equation}
    S\cap B_1(0^n)\subset \bigcup_{c\in B_{1}(0^k)}\iota(c).
\end{equation}
\end{rem}

\begin{rem}\label{main theorem general version}

The first condition for $(k,\delta,N)$-splitting Reifenberg sets is meant to address boundary issues. We actually have the following analog of the main theorem without that condition:

Let $S\subset B_2(0^n)$ be a closed set. Assume for each $B_r(x)\subset B_2(0^n)$, $S$ is a $(k,\delta)$-splitting with splitting set $A_{x,r}$, such that $|A_{x,r}^\perp|\le N$. Then there exists a map $\iota:B_{1}(0^k)\to (2^{S},d_H)$, such that for each $c,d\in B_{1}(0^k)$ with $|c-d|\le 1$,
\begin{equation}
    (1-C(n,N)\delta)|c-d|^{\frac{1}{1-\alpha}}\le d_H|_{B_1(0^n)}(\iota(c),\iota(d))\le (1+C(n,N)\delta)|c-d|^{1-\alpha},
\end{equation}
such that 
\begin{equation}
S\cap B_1(0^n)\subset \iota(B_1(0^k)).
\end{equation}
    
\end{rem}

\begin{rem}
In the above theorems, if we assume that for each $x\in B_{1.99}(0^n)$, there are at most $N'$ scales $s_j=2^{-j}$ such that the splitting set $A_{x,s_j}$ is not a single $k$-plane, then we can take the map $\iota$ in the main theorem, such that $|\iota(c)|\le C(n,N')$ for all $c\in B_1(0^k)$.
\end{rem}

The remaining part of this article will be divided into three chapters, with the final goal to prove Theorem \ref{main theorem}. As a one line introduction for each section: Section 2 will be general preliminaries and Section 3 will be the explanation of $(k,\delta,N)$-splitting Reifenberg sets. Section 4 will be the main part of the proof, where we construct a H\"older map $\Phi:\Rb^n\to \Rb^k$, take $\iota(c)$ as $\Phi^{-1}(c)\cap S$ and prove all the estimates.

\section{Preliminaries}

This article aims to be self contained. This section covers the notions and preliminaries that will be used later, including QR decomposition, Hausdorff distances, regular maps and and graphical submanifolds. Readers may feel free to skip this section and check it back when necessary.

\subsection{Derivatives and Norms}

This subsection recalls the derivatives of vector valued functions.

\begin{dfn}\label{derivative}

\bnu

\item Let $f:\Rb^n\to \Rb^k$ be a smooth map. The $h$-th derivative of $f$ at $x$ is a multilinear map $\nabla^h f(x):(\Rb^n)^h\to \Rb^k$,
\begin{equation}\nabla^hf(x)[v_1,\dots,v_h]=\partial_{v_1}\dots\partial_{v_h}f(x).\end{equation}

In particular, the derivative $\nabla f$ is a linear map $\Rb^n\to \Rb^k$, or a matrix is $M_{k\times n}(\Rb)$.

\item Let $f:\Rb^n\to M_{k\times l}(\Rb)$ be a matrix valued smooth map. The $h$-th derivative of $f$ at $x$ is a multilinear map: $\nabla^hf(x):(\Rb^n)^h\to M_{k\times l}(\Rb)$, \begin{equation}\nabla^hf(x)[v_1,\dots,v_h] = \partial_{v_1}\dots\partial_{v_h}f(x).\end{equation}

If we view $\nabla f(x)$ as a matrix then this definition is compatible with 1..

\item Let $f:\Rb^n\to \Rb^k$ be a smooth map and $B\in M_{m\times k}(\Rb)$. We may multiply a $B$ to $\nabla^hf$,
\begin{equation}(B\nabla^hf)(x)[v_1,\dots,v_h]:=B(\nabla^hf(x)[v_1,\dots,v_h])= \nabla^h(Bf)(x)[v_1,\dots,v_h].\end{equation}

If $f:\Rb^n\to M_{k\times l}(\Rb^n)$ is matrix valued, we may also multiply $B$ to $\nabla^hf$ the same way by matrix multiplication.

\enu

\end{dfn}

We'll always use $2$-norms for vectors, matrices and tensors, unless stated otherwise explicitly. And we will use the $L^\infty$ norm for maps. The $2$ and $L^\infty$ subscript will sometimes be omitted for simplicity.

\begin{dfn}\label{norm}

\bnu 

\item Let $M=(m_{ij})\in M_{k\times n}(\Rb)$ be a matrix.\begin{equation}|M|_2=(\sum_{i,j}m_{ij}^2)^{1/2}.\end{equation}
\item Let $f:\Rb^n\to\Rb^k$ be a smooth map. For $x\in \Rb^n$, the norm of $\nabla^hf$ at $x$ is \begin{equation}|\nabla^h f(x)|_2=(\sum_{i_1,\dots,i_h,j}\partial_{i_1}\dots\partial_{i_h}f_j(x)^{2})^{1/2}.\end{equation}
\item Assume $f:\Rb^n\to M_{k\times l}(\Rb)$ is matrix valued, the norm of $\nabla^hf$ at $x$ is
\begin{equation}|\nabla^h f(x)|_2=(\sum_{i_1,\dots,i_h,j,j'}\partial_{i_1}\dots\partial_{i_h}f_{jj'}(x)^{2})^{1/2},\end{equation}
where $f_{jj'}$ is the matrix entry.
\item Let $f:\Rb^n\to\Rb^k$ be a smooth map. The $L^\infty$ norm of $\nabla^hf$ is 

\begin{equation}||\nabla^h f||_{L^{\infty}}=\sup_x |\nabla^h f(x)|_2.\end{equation}
We may define the norm $||\nabla^hf||_{L^\infty}$ for matrix valued $f$ similarly.

\enu
  
\end{dfn}

The reason we choose these norms is that they are invariant under orthogonal transformations.

\begin{lem}\label{invariance of norm}

\bnu 

\item Let $M\in M_{k\times n}(\Rb)$ be a matrix. $O_1\in O(k)$, $O_2\in O(n)$ be orthogonal matrices. Then 
\begin{equation}|O_1M|=|MO_2|=|M|.\end{equation}
\item Let $f:\Rb^n \to \Rb^k$ be a smooth map, $O_1\in O(k)$, $O_2\in O(n)$. Then
\begin{equation}|O_1\nabla^h f(x)| = |\nabla^h f(x)|,\end{equation}
\begin{equation}|\nabla^h(f\circ O_2)(x)| = |\nabla^h f(O_2x)|.\end{equation}
The same is also true if $f$ is matrix valued.

\enu

\pf

\bnu

\item This is a standard result in linear algebra.

\item The first equality holds for any tensors with the same form. The second follows from the fact that
\begin{equation}\nabla^h(f\circ O_2)(x)[v_1,\dots,v_h]=\nabla^hf(O_2x)[O_2v_1,\dots,O_2v_h].
\end{equation}
\enu
\epf
\begin{rem}

Lemma \ref{invariance of norm} implies that $||\nabla^h f||$ is independent to the choice of orthogonal bases on $\Rb^n$ and $\Rb^k$.
    
\end{rem}

\end{lem}

\subsection{QR Decomposition}

This subsection recalls the QR decomposition for rectangular matrices.

We'll mostly concern matrices with full rank in this article. Among them, the distinguished ones are Riemannian submersions.
\begin{dfn}

For $n>k$, $\pi\in M_{k\times n}(\Rb)$, we say $\pi$ is a Riemannian submersion, denoted as $\pi\in O_{k\times n}(\Rb)$, if rows of $\pi$ form an orthogonal system.

\end{dfn}

\begin{rem}

For $\pi\in M_{k\times n}(\Rb^n)$, $\pi\in O_{k\times n}(\Rb)$ if and only if $\pi$ is isometric on $(\ker \pi)^\perp$ as a linear map.

\end{rem}

Now we state the QR decomposition for rectangular matrices, which is a standard result in linear algebra. It says that any rectangular matrix with full rank can be written uniquely as the product of a lower rectangular matrix and a Riemannian submersion.

\begin{prop}\label{QR}

For $n\ge k$, take $M\in M_{k\times n}(\Rb)$ with $rank(M)=k$. Then there is a unique decomposition $M=L(M)\pi(M)$ where $L(M)\in GL_{k}(\Rb)$ is lower triangular with positive diagonal entries and $\pi(M)\in O_{k\times n}(\Rb)$.

\end{prop}

$L$ and $\pi$ are clearly smooth maps on the space of matrices with full rank, since they are obtained from the Gram-Schmidt orthonormalization. The following lemma gives a $C^1$ control of them modulo $L$.

\begin{lem}\label{smoothness of QR}

There is a $C(n)>0$ such that
\begin{equation}
||L^{-1}\nabla L\circ \Lc_{L}||_{L^\infty}\le C(n),\; ||\nabla \pi\circ \Lc_{L}||_{L^\infty}\le C(n)
\end{equation}
in the space of $k\times n$ matrices with full rank. Here $\Lc_L$ is the operator of left multiplying $L$, i.e. 
\begin{equation}(L^{-1}\nabla L\circ \Lc_{L})(M)[B]=L(M)^{-1}\nabla L(M)[L(M)B].\end{equation}
    
\end{lem}

\pf

Note that 
\begin{align}\begin{split}
L(M)^{-1}\nabla L(M)\circ \Lc_{L(M)}L(M)[B] &= L(M)^{-1}\left.\frac{\drm}{\drm t}\right\rvert_{t=0}L(M+tL(M)B)\\
&= \left.\frac{\drm}{\drm t}\right\rvert_{t=0}L(\pi(M)+tB)\\
&=\nabla L(\pi(M))[B].
\end{split}\end{align}
\begin{align}\begin{split}
\nabla \pi(M)\circ \Lc_{L(M)}[B] &= \dvert{t}{0}\pi(M+tL(M)B)\\
&= \dvert{t}{0}\pi(\pi(M)+tB)\\
&=\nabla \pi(\pi(M))[B].    
\end{split}\end{align}

Note that $\pi(M)\in O_{k\times n}(\Rb)$, the bounds follow immediately from the compactness of $O_{k\times n}(\Rb)$.

\epf

\begin{lem}\label{smoothness of QR cor1}
For $A,B\in M_{k\times n}(\Rb)$, assume $rank(A)=k$ and
\begin{equation}
    |L(A)^{-1}(A-B)|\le \delta
\end{equation}
for $\delta<\delta(n)$ sufficiently small. Then $B\in M^*_{k\times n}(\Rb)$ and
\begin{equation}
    |L(A)^{-1}L(B)-I_k|\le C(n)\delta,\;|\pi(A)-\pi(B)|\le C(n)\delta.
\end{equation}
\end{lem}

\pf
Let $B'=L(A)^{-1}B$. Then $|\pi(A)-B'|\le \delta$. We have $L(B')=L(A)^{-1}L(B)$, $\pi(B')=\pi(B)$. It suffices to prove
\begin{equation}
    |L(B')-I_k|\le C(n)\delta, \;|\pi(B')-\pi(A)|\le C(n)\delta.
\end{equation}
Note that $L$ and $\pi$ are smooth in $B_{C(n)\delta}(O_{k\times n}(\Rb))$, which is a compact neighborhood of $O_{k\times n}(\Rb)$. $||\nabla L||_{L^{\infty}}$, $||\nabla \pi||_{L^\infty}\le C(n)$ in this neighborhood. Thus the conclusion is proved.
\epf

\subsection{Hausdorff Distance}

In this subsection we'll introduce the Hausdorff distance.

The Hausdorff distance is the analog of $L^\infty$ distance for subsets of a metric space. For simplicity we only concern subsets of $\Rb^n$.

\begin{dfn}
Let $A,B\subset \Rb^n$. The Hausdorff distance between $A$ and $B$ is
\begin{equation}
d_H(A,B):=\inf\{s>0\,|\,A\subset B_s(B), B\subset B_s(A)\}. 
\end{equation}
\end{dfn}

Define the space of closed and finite subsets of $\Rb^n$ as following. 

\begin{dfn}

Let $2^{\Rb^n}$ be the space of closed subsets of $\Rb^n$. Let $2^{\Rb^n}_N$ be the space of subsets of $\Rb^n$ that contains at most $N$ points.
    
\end{dfn}

It is well known that $d_H$ is indeed a metric.
\begin{lem}
$d_H$ is a metric on $2^{\Rb^n}$ and $2^{\Rb^n}_N$.
\end{lem}

Actually, these spaces are complete under $d_H$.

\begin{lem}
$(2^{\Rb^n},d_H)$ and $(2^{\Rb^n}_N,d_H)$ are complete.
\end{lem}
\pf
First consider the space $(2^{\Rb^n},d_H)$. Let $A_i\subset \Rb^n$ be a sequence of closed sets that is Cauchy in $d_H$. By taking a subsequence we may assume $d_H(A_i, A_{i+1})\le 2^{-i-1}$. Let
\begin{equation}
    A:=\{x\in \Rb^n\,|\,d(x,A_i)\to 0\}.
\end{equation}
Then $A$ is clearly closed. It suffices to check 
\begin{equation}
    d_H(A_i,A)\le 2^{-i}\text{ for each }i.
\end{equation}
For any $x_i\in A_i$, there exists $x_{i+1},x_{i+2},\dots$ with $x_{i+j}\in A_{i+j}$ and $d(x_{i+j},x_{i+j+1})\le 2^{-i-j-1}$. Thus $\{x_{i+j}\}$ converges to some $x\in A$ and $d(x_i,A)\le 2^{-i}$. Conversely, for any $x\in A$ and any $\epsilon>0$, there is some $x_{i+k}\in A_{i+k}$ with $d(x,x_{i+k})< \epsilon$. Similarly, there exists $x_i, x_{i+1},\dots,x_{i+k-1}$ with $x_{i+j}\in A_{i+j}$, $d(x_{i+j},x_{i+j+1})\le 2^{-i-j-1}$. Thus $d(x,A_i)< 2^{-i}+\epsilon$. Let $\epsilon \to 0$ we get $d(x,A_i)\le 2^{-i}$. Thus $d_H(A_i,A)\le 2^{-i}$ and $A_i\to A$ in $d_H$.

For the space $2^{\Rb}_N$, we only need to show it is closed in $d_H$. Assume $A_i\to A$ with $|A_i|\le N$. If $|A|\ge N+1$, pick $x_1, \dots,x_{N+1}$ in $A$. There is $\epsilon>0$ such that $d(x_j,x_k)>\epsilon$ for any $j,k$. There exists some $A_i$ with $d(A_i,A)<\epsilon/2$. Then $|A_i|\ge N+1$ because it contains a point in $B_{\epsilon/2}(x_j)$ for each $j$. This is a contradiction. Thus $|A|\le N$ and the proof is finished.
\epf

Aside from the Hausdorff distance, let's recall the definition for the distance between subsets of $\Rb^n$.
\begin{dfn}
    For two subsets $A,B\subset \Rb^n$, define their distance to be
    \begin{equation}
        dist(A,B):=\inf\{d(a,b)\,|\,a\in A,b\in B\}.
    \end{equation}
\end{dfn}
\begin{rem}
    It is clear that $dist(A,B)\le d_H(A,B)$.
\end{rem}

\subsection{Local Hausdorff Distance}

In this subsection we'll introduce the local Hausdorff distance.

The definition is an analog of the Hausdorff distance.
\begin{dfn}[local Hausdorff distance]

For $A,B\subset\Rb^n$, $B_r(x)\subset \Rb^n$. \begin{equation}d_H|_{B_r(x)}(A,B):=\inf\{s>0:A\cap B_r(x)\subset B_s(B),B\cap B_r(x)\subset B_s(A)\}.\end{equation}
    
\end{dfn}

One should beware that $d_H|_{B_r(x)}$ is not a metric, because it does not have subadditivity.

\begin{expl}

For small $\epsilon>0$, define
\begin{equation}
A=\{0,1-\epsilon\}, \; B=\{0,1\},\; C=\{0,1+\epsilon\}\subset \Rb.
\end{equation}
Then
\begin{equation}
d_H|_{B_r(x)}(A,B)=\epsilon,\;d_H|_{B_r(x)}(B,C)=0,\;d_H|_{B_r(x)}(A,C)=2\epsilon.
\end{equation}
Thus
\begin{equation}
    d_H|_{B_r(x)}(A,B)+d_H|_{B_r(x)}(B,C)<d_H|_{B_r(x)}(A,C).
\end{equation}
\end{expl}

However, we will still primarily use the local Hausdorff distance in this article, since it passes to subsets more easily. The following examples shows that the Hausdorff distance could be inconvenient when we pass to a subset.

\begin{expl}

For small $\epsilon>0$, define
\begin{equation}
A=\{0,1-\epsilon\}, \; B=\{0,1+\epsilon\}\subset \Rb.
\end{equation}
Then
\begin{equation}d_H(A\cap B_1(0),B\cap B_1(0))=1-\epsilon,\end{equation}
\begin{equation}d_H|_{B_1(x)}(A,B)=2\epsilon.\end{equation}
    
\end{expl}

In the above example, we expect $A$ and $B$ to be Hausdorff close locally. This is only observed in the local Hausdorff distance.

The following lemma summarizes the basic properties of the local Hausdorff distance, whose proofs are standard.

\begin{lem}\label{local Hausdorff distance properties}

The local Hausdorff distance has the following properties.

\bnu

\item $d_H|_{B_r(x)}(A,B)\le \delta$ if and only if for any $a\in A\cap B_r(x)$, $b\in B\cap B_r(x)$, we have 
\begin{equation}
    d(a,B)\le \delta,\;d(b,A)\le \delta.
\end{equation}

% \item If $d_H|_{B_r(x)}(A,B)\le \delta$, then $d_H|_{B_r(x)}(A^\perp,B)\le \delta$.

\item If $d_H|_{B_r(x)}(A,B)\le \delta_1$, $d_H|_{B_r(x)}(B,C)\le \delta_2$ with $\delta_1\le\delta_2$, then 
\begin{equation}
    d_H|_{B_{r-\delta_2}(x)}(A,C)\le \delta_1+\delta_2.
\end{equation}
\item $d_H|_{B_r(x)}(A,B)\le d_H(A\cap B_r(x),B\cap B_r(x))$.
\item If $B_r'(x')\subset B_r(x)$, then
\begin{equation}
    d_H|_{B_r'(x')}(A,B)\le d_H|_{B_r(x)}(A,B).
\end{equation}
\enu
\end{lem}

\subsection{Distance between $k$-subspaces}

In this subsection we will define a distance between $k$-subspaces of $\Rb^n$, or equivalently a metric on the Grassmannian $\mathbf{Gr}(k,n)$. Many different definitions exist in literature and any of them will probably work for the use of this article. For the cleanness of the result, we'll choose the one that generalizes the local Hausdorff distance.

\begin{dfn}\label{metric on Grassmannian}

Let $\hat{l}_1,\hat{l}_2$ be $k$-linear subspaces of $\Rb^n$, their distance is defined to be

\begin{equation}d(\hat{l}_1,\hat{l}_2):=d_H|_{B_1(0^n)}(\hat{l}_1,\hat{l}_2)=\max\{\sup_{x_1\in \hat{l}_1\cap B_1(0^n)} d(x_1,\hat{l}_2), \sup_{x_2\in \hat{l}_2\cap B_1(0^n)} d(x_2,\hat{l}_1)\}.\end{equation}
    
\end{dfn}

Despite that the local Hausdorff distance is not a metric, the distance between $k$-subspaces is indeed a metric.

\begin{lem}

The $d$ in Definition \ref{metric on Grassmannian} is a metric on $\mathbf{Gr}(k,n)$.

\end{lem}

\pf

We only need to check subadditivity. Let $\hat{l}_1$, $\hat{l}_2$, $\hat{l}_3$ be $k$-subspaces of $\Rb^n$. For any $x\in \hat{l}_1\cap B_1(0^n)$, there exists $x_2\in \hat{l}_2$ with \begin{equation}|x_1-x_2|\le d(\hat{l}_1,\hat{l}_2).\end{equation}

Since $|x_1|^2=|x_1-x_2|^2+|x_2|^2$, we have $x_2\in B_1(0^n)$. Similarly, there is an $x_3\in B_1(0^n)\cap\hat{l}_3$ such that $d(x_2,x_3)\le d(\hat{l}_2,\hat{l}_3)$. Thus $d(x_1,\hat{l}_3)\le d(\hat{l}_1,\hat{l}_2)+d(\hat{l}_2,\hat{l}_3).$ Therefore
\begin{equation}d(\hat{l}_1,\hat{l}_3)\le  d(\hat{l}_1,\hat{l}_2)+d(\hat{l}_2,\hat{l}_3).\end{equation}
\epf

There are a couple pleasant ways to rewrite the definition.

\begin{lem}\label{equivalent definition of distance of subspaces}

Let $\hat{l}_1,\hat{l}_2$ be $k$-linear subspaces of $\Rb^n$.
\bnu
\item The maximal distance is attained by both expressions.
\begin{equation}
d(\hat{l}_1,\hat{l}_2)=\sup_{x_1\in \hat{l}_1\cap B_1(0^n)} d(x_1,\hat{l}_2)=\sup_{v\in \hat{l}_1\cap \Sb^{n-1}} d(v,\hat{l}_2).
\end{equation}
\item Let $P_1$, $P_2$ be the orthogonal projections to $\hat{l}_1$, $\hat{l}_2$. Then
\begin{equation}
d(\hat{l}_1,\hat{l}_2)=|P_1-P_2|_{op}.
\end{equation}

\enu

\end{lem}

\pf

\bnu

\item This will be proved together with 2.

\item We'll prove that $|P_1-P_2|_{op}= \sup_{v\in \Sb^{n-1}\cap \hat{l}_1} d(v,\hat{l}_2)$. By symmetry this will imply 1. and 2.. It is clear that $|P_1-P_2|_{op}\ge \sup_{v\in \Sb^{n-1}\cap \hat{l}_1} d(v,\hat{l}_2)$. Thus we only need to prove the other direction.

Take any unit vector $v\in \Rb^n$, $v = v_1 + v_1^\perp$ where $v_1\in \hat{l}_1$, $v_1^\perp\in \hat{l}_1^\perp$. Then
\begin{equation}
P_1(v)-P_2(v)=(v_1-P_2(v_1))-P_2(v_1^\perp)
\end{equation}
with $(v_1-P_2(v_1))\in \hat{l}_2^\perp$, $P_2(v_1^\perp)\in \hat{l}_2$. Thus
\begin{align}\begin{split}
    |P_1(v)-P_2(v)|^2&=|(v_1-P_2(v_1))|^2+|P_2(v_1^\perp)|^2\\
    &\le |\left.(\id-P_2)\right\rvert_{\hat{l}_1}|_{op}^2|v_1|^2+|\left.P_2\right\rvert_{\hat{l}_1^{\perp}}|_{op}^2|v_1^{\perp}|^2\\
    &\le \max\{|\left.(\id-P_2)\right\rvert_{\hat{l}_1}|_{op},|\left.P_2\right\rvert_{\hat{l}_1^{\perp}}|_{op}\}^2
\end{split}\end{align}
Note that
\begin{equation}|\left.(\id-P_2)\right\rvert_{\hat{l}_1}|_{op}=|\left.P_2\right\rvert_{\hat{l}_1^{\perp}}|=\sup_{v\in \Sb^{n-1}\cap \hat{l}_1} d(v,\hat{l}_2),\end{equation}
we have $|P_1-P_2|_{op}\le \sup_{v\in \Sb^{n-1}\cap \hat{l}_1} d(v,\hat{l}_2)$.
\enu
\epf

The distance between subspaces is invariant after taking the orthogonal complements.

\begin{lem}

Let $\hat{l}_1,\hat{l}_2$ be $k$-linear subspaces of $\Rb^n$ and $\hat{l}_1^\perp$, $\hat{l}_2^\perp$ be their orthogonal complements, then
\begin{equation}d(\hat{l}_1,\hat{l}_2)=d(\hat{l}^\perp_1,\hat{l}_2^\perp).\end{equation}

\end{lem}

\pf

Let $d_{\Sb^{n-1}}$ be the spherical distance on the unit sphere $\Sb^{n-1}$. For any $v\in\hat{l}_1\cap \Sb^{n-1}$, $d(v,\hat{l}_2)+d(v,\hat{l}_2^\perp)^2=1$. Thus 
\begin{equation}
d(v,\hat{l}_2)=\cos d_{\Sb^{n-1}}(v,\hat{l}_2^\perp\cap \Sb^{n-1}).
\end{equation}
Taking the supremum for both sides, we get
\begin{equation}
d(\hat{l}_1,\hat{l}_2)=\cos d_{\Sb^{n-1}}(\hat{l}_1\cap \Sb^{n-1},\hat{l}_2^\perp\cap \Sb^{n-1}).
\end{equation}
Similarly, we know
\begin{equation}
\cos d_{\Sb^{n-1}}(\hat{l}_1\cap \Sb^{n-1},\hat{l}_2^\perp\cap \Sb^{n-1})=d(\hat{l}_1^\perp,\hat{l}_2^\perp).
\end{equation}
Now the conclusion is immediate.
\epf

\subsection{Regular Maps}

We'll discuss $\delta$-regular maps in the rest of this section. In this subsection, we'll give the definition of $\delta$-regular maps. And later we'll prove some properties for their level sets.

A map $f$ is $\delta$-regular if its differentiation is surjective and its Hessian is $\delta$ small modulo the linear transformation $L(\nabla f)$ rising from the QR decomposition of $\nabla f$.
\begin{dfn}

Let $U\subset \Rb^n$ be an open subset and $f:\Rb^n\to \Rb^k$ be a smooth map such that $\nabla f(x)$ is surjective for all $x\in U$. We call $f$ a $\delta$-regular map in $U$ if $||L(\nabla f)^{-1}\nabla^2 f||_{L^{\infty}(U)}\le \delta$.
    
\end{dfn}

In the above definition, the linear transformation $L(\nabla f)$ is taken pointwise. Since $f$ is regular, $L(\nabla f)$ is almost constant locally. Thus we may replace it by a fixed transformation $L_0$ in a small ball. The following is the main lemma of this section:

\begin{lem}\label{regular map equivalent definition}

Assume $B_r(x_0)\subset \Rb^n$ and $\delta<\epsilon<\epsilon(n)$ are sufficiently small. Let $f:\Rb^n\to \Rb^k$ be a smooth map in $B_{r}(x_0)$ with $rank(\nabla f(x_0))=k$. Take $L_0\in GL_k(\Rb^n)$ such that $|L(\nabla f(x_0))^{-1}L_0 - I_k|<\epsilon$. Then

\bnu
\item If $f$ is $\delta r^{-1}$-regular in $B_r(x_0)$, then $||L_0^{-1}\nabla^2 f||_{L^\infty(B_r(x_0))}\le (1+C(n)\epsilon)\delta r^{-1}$.
\item If $||L_0^{-1}\nabla^2 f||_{L^\infty(B_r(x_0))}\le \delta r^{-1}$, then $f$ is $(1+C(n)\delta)\delta r^{-1}$-regular in $B_r(x_0)$. 
\enu
\end{lem}

The proof of Lemma \ref{regular map equivalent definition} is straightforward but we need a few technical steps.

The following technical lemmas control the growth of a matrix-valued function $f$ assuming that $||f^{-1}\nabla f||\le \delta$. We first consider the case on a segment.

\begin{lem}\label{matrix derivative}

Let $f:[0,1]\to GL_{k}(\Rb)$ be a smooth matrix-valued map. Given $\delta<1/2$, assume $|(f(t))^{-1}f'(t)|\le \delta$ for all $t\in [0,1]$. Then we have
\begin{equation}|f(0)^{-1}f(1)-I_k|\le 2\delta.\end{equation}

\end{lem}

\pf

This can be proved by continuity method. Let

\begin{equation}s=\inf\{t\in[0,1]\mid|f(0)^{-1}f(t)-I_k|>2\delta\}.\end{equation}

By continuity of $f$, $s>0$. It suffices to prove $s = 1$. If not, then $|f(0)^{-1}f(t)-I_k|=2\delta$. Note that

\begin{equation}f(s)=f(0)+\int_{0}^sf'(t)\drm t.\end{equation}
We have
\begin{align}\begin{split}
    |f(0)^{-1}f(s)-I_k|&=|\int_0^sf^{-1}(0)f(t)\cdot f(t)^{-1}f'(t)\drm t|\\
    &\le(1+2\delta)\int_0^s|f(t)^{-1}f'(t)|\drm t\\
    &\le(1+2\delta)\delta s\\
    &<2\delta.
\end{split}\end{align}

A contradiction! Hence $s=1$ and $|f(0)^{-1}f(1)-I_k|\le 2\delta.$

\epf

\begin{cor}\label{MVT for matrix valued functions}

Let $f:\Rb^n\to GL_{k}(\Rb)$ be a smooth map. Assume $|f(x)^{-1}\nabla f(x)|\le \lambda$ for all $x\in \Rb^n$. Then for $y\in B_{1/(2\lambda)}(x)$, 
\begin{equation}|f(x)^{-1}f(y)-I_k|\le 2\lambda |x-y|.\end{equation}

\end{cor}

\pf

Consider the constant-speed segment $\gamma:[0,1]\to\overline{xy}$ from $x$ to $y$. Define $g(t)=f(\gamma(t))$. We have $|g(t)^{-1}g'(t)|\le \lambda|x-y|$. Apply Lemma \ref{matrix derivative} to $g$ and the conclusion follows.

\epf

\begin{lem}\label{bounding nabla L(f) changes by nabla f}

Let $f:\Rb^n\to GL_k(\Rb^n)$ be a smooth map. Assume $rank(\nabla f(x))=k$ and 
$|L(f(x))^{-1}\nabla f(x)|\le \delta.$ Then
\begin{equation}|L(f(x))^{-1}\nabla(L(f))(x)|\le C(n)\delta.\end{equation}

\end{lem}

\pf
This estimate is purely pointwise. By Lemma \ref{smoothness of QR},
\begin{align}\begin{split}
    |L(f(x))^{-1}\nabla(L(f))(x)|&=|L(f(x))^{-1}(\nabla L)(f(x))\nabla f(x)|\\
    &=|(L(f(x))^{-1}(\nabla L)(f(x))\circ \Lc_{L(f(x))})\circ(L(f(x))^{-1}\nabla f(x))|\\
    &\le C(n)|L(f(x))^{-1}\nabla f(x)|\\
    &\le C(n)\delta.
\end{split}\end{align}

\epf

\pf[Proof of Lemma \ref{regular map equivalent definition}]
\bnu

\item Assume $f$ is $\delta r^{-1}$-regular in $B_r(x_0)$. By Lemma \ref{bounding nabla L(f) changes by nabla f},  we have 
\begin{equation}
    |L(\nabla f(x))^{-1}\nabla(L(\nabla f))(x)|\le C(n)\delta r^{-1}.
\end{equation}
By Corollary \ref{MVT for matrix valued functions}, if $\delta <\epsilon(n)$ is sufficiently small, then for any $x\in B_{r}(x_0)$ we have
\begin{equation}\label{L(nabla f) is almost locally constant}
    |L(\nabla f(x_0))^{-1}L(\nabla f(x))-I_k|\le C(n)\delta.
\end{equation}

Note that $|L(\nabla f(x_0))^{-1}L_0 - I_k|<\epsilon$, we have
\begin{equation}
    |L_0^{-1}L(\nabla f(x))-I_k|< C(n)\epsilon.
\end{equation}
Thus by the $\delta$-regularity of $f$, we have
\begin{equation}
    ||L_0^{-1}\nabla^2f||_{L^\infty(B_r(x_0))}\le (1+C(n)\epsilon)\delta r^{-1}.
\end{equation}

\item The proof of this part is simpler. The conditions implies
\begin{equation}
    ||L(\nabla f(x_0))\nabla^2 f||_{L^\infty(B_r(x_0))}\le (1+C(n)\epsilon)\delta r^{-1}.
\end{equation}

Taking integral, we get
\begin{equation}
    |L(\nabla f(x_0))^{-1}(\nabla f(x)-\nabla f(x_0))|\le (1+C(n)\epsilon)\delta .
\end{equation}

Thus if $\delta<\delta(n)$, by Lemma \ref{smoothness of QR cor1}, we have
\begin{equation}
    |L(\nabla f(x_0))^{-1}L(\nabla f(x))-I_k|\le C(n)\delta.
\end{equation}
and $|L_0^{-1}L(\nabla f(x)) - I_k|\le C(n)\epsilon$.

Thus the condition implies $||L^{-1}\nabla^2f||_{L^\infty(B_r(x_0))}\le (1+C(n)\epsilon)\delta r^{-1}$, i.e., $f$ is $(1+C(n)\epsilon)\delta r^{-1}$-regular in $B_r(x_0)$.

\enu

\epf

\begin{rem}

Assume either case of Lemma \ref{regular map equivalent definition}. From the proof we see that 
\begin{equation}
    ||L_0^{-1}L(\nabla f)-I_k||_{L^\infty(B_r(x_0))}\le C(n)\epsilon.
\end{equation}
    
\end{rem}

\subsection{Graphical Submanifolds}

In this subsection, we'll introduce $(\delta,r)$-submanifolds and prove that level sets of regular maps are graphical submanifolds.

We call a $k$-dimensional submanifold of $\Rb^n$ graphical if it is locally the graph of some $C^2$ map that is almost constant.

\begin{dfn}[$(\delta,r)$-graphical submanifold]\label{graphical submanifold}

Let $S$ be a subset of $\Rb^n$. We say $S$ is a $k$-dimensional $(\delta,r)$-graphical if for any $x\in S$, there exists a $k$-plane $l_x\in \Rb^n$, called graphing plane, and a smooth graphing function
\begin{equation}
    g_x:l_x\to l_x^\perp\text{ with }||g_x||\le \delta r, ||\nabla g_x||\le\delta, ||\nabla^2g_x||\le \delta r^{-1},
\end{equation}
such that 
\begin{equation}
    S\cap B_r(x)=Graph(g_x)\cap B_r(x).
\end{equation}
\end{dfn}

\begin{rem}

It is clear that $k$-dimensional $(\delta,r)$-graphical submanifolds are smooth embedded $k$-dimensional submanifolds without boundary.
    
\end{rem}

For a $(\delta,r)$-graphical submanifold $S$, we always use $l_x$ to denote a graphing plane at $x$. Let $P_{l_x}:\Rb^n\to l_x$ be the orthogonal projection to $l_x$ and $\hat{P}_{l_x}=\nabla P_{l_x}$ be the orthogonal projection to the respective linear subspace.

The following is the main result to prove in this subsection.

\begin{prop}\label{level sets of delta_x regular maps}
For $\delta<\delta(n)$, $r>0$, assume $f:\Rb^n\to \Rb^k$ is a $\delta r^{-1}$-regular map. Then for any $c\in \Rb^k$, $f^{-1}(c)$ is a $(C(n)\delta,r)$-graphical submanifold, with graphing plane $l_x=x+\ker\nabla f(x)$ at $x$.
    
\end{prop}

Before we prove Proposition \ref{level sets of delta_x regular maps}, let's first state the following effective Inverse and Implicit Function Theorems. The rest of this section are essentially the applications of them.

\begin{prop}[Quantitative Inverse Function Theorem]\label{quantitative inverse function theorem}
    Let $f:B_1(0^n)\to \Rb^n$ be a smooth map with $|f(0^n)|,|\nabla f(0^n)-I_n|$, $||\nabla^2f||_{L^\infty(B_1(0^n))}\le \delta$ for some $\delta<\delta(n)$. Then $f$ is a diffeomorphism from $B_1(0^n)$ onto its image, with $B_{1-C'(n)\delta}(0^n)\subset f(B_1(0^n))$. Moreover, we have $||f^{-1}-\id||, ||\nabla f^{-1} - I_n||, ||\nabla^2f^{-1}||\le C(n)\delta$ in $B_{1-C'(n)\delta}(0^n)$.
\end{prop}

\pf

By the conditions, we know $||\nabla f-I_n||\le C(n)\delta$ in $B_1(0^n)$. Thus for any $x,y\in B_1(0^n)$, $|f(x)-x|\le C(n)\delta$, and $|f(y)-f(x)|\ge (1-C(n)\delta)|x-y|$. Thus $f$ is injective in $B_1(0^n)$. The classical inverse function theorem implies that $f$ is a diffeomorphism onto its image with $\nabla (f^{-1})(x)=(\nabla f(f^{-1}(x)))^{-1}$. Thus all the estimates for $f^{-1}$ holds.

It suffices to show that $B_{1-C'(n)\delta}(0^n)\subset f(B_1(0^n))$ for some $C'(n)$. For $y\in B_{1-C'(n)\delta}(0^n)$, define
\begin{equation}
    g:B_1(0^n)\to \Rb^n, x\mapsto |f(x) - y|.
\end{equation}
For $C'(n)$ sufficiently large, $g$ attains minimum at some interior point $x_0\in B_1(0^n)$. Since $|\nabla f(x_0) -I_n|\le C(n)\delta$, we must have $f(x_0)=y$.
\epf

\begin{prop}[Quantitative Implicit Function Theorem]\label{quantitative implicit function theorem}

Let $f:B_1(0^{n-k})\times B_1(0^k)\to \Rb^k$ be a smooth function. Assume $|f(0^{n-k},0^k)|\le \delta$, $|\nabla f(0^{n-k},0^{k})-(0,I_k)|\le \delta$, $||\nabla^2f||\le \delta$, where $\delta <\delta(n)$ is sufficiently small. Then for any $x\in B_1(0^{n-k})$ there exists a unique $y_x\in B_1(\Rb^k)$ such that $f(x,y_x)=0$. And the map $g:B_1(0^{n-k})\to \Rb^k$, $x\mapsto y_x$ is smooth with $||g||$, $||\nabla g||$, $||\nabla^2 g||\le C(n)\delta$.
    
\end{prop}

\pf

The conditions implies that $|f(x,0^k)|\le C\delta$ for any $x\in B_1(0^{n-k})$. Apply Proposition \ref{quantitative inverse function theorem} to each $\{x\}\times \Rb^k$. We know there is a unique $y_x\in B_{1}(0^k)$ such that $f(x,y_x)=0$, and we have $|y_x|\le C(n)\delta$. The smoothness and controls of $y_x$ follows from the classical implicit function theorem.
\epf

\begin{rem}
    There are $C^1$ versions of the quantitative inverse and implicit function theorems: If we only assume $|f(0^n)|$, $||\nabla f-I_n||\le \delta$ and $|f(0^{n-k},0^k)|$, $||\nabla f-(0,I_k)||\le \delta$ respectively, the conclusions are still true, except that we won't get the $C^2$ estimates.
\end{rem}

Let's go back to Definition \ref{graphical submanifold}. In the definition, the exact domain of $g_x$ doesn't matter so much, since only the part $Graph(g_x)\cap B_r(x)$ matters. The following remark shows that it suffices to construct the $g_x$ for this part.

\begin{rem}\label{continuous extension}
Assume we have some $g_x:\overline{U}_x\to l_x^\perp$ for some $U_x\subset l_x$, with $Graph(g_x|_{\partial U_x})\subset \partial B_r(x)$, $Graph(g_x)\subset \overline{B}_r(x)$. Assume $||g_x||\le \delta r, ||\nabla g_x||\le\delta, ||\nabla^2g_x||\le \delta r^{-1}$. Then $U_x$ has a smooth boundary and $g_x$ may be extended from $\overline{U}_x$ to $l_x$ with 
\begin{equation}
    ||g_x||\le C(n)\delta r, ||\nabla g_x||\le C(n)\delta, ||\nabla^2g_x||\le C(n)\delta r^{-1}.
\end{equation}
\end{rem}

The graphing plane $l_x$ is almost unique. It can also be perturbed a bit or directly taken as $T_xS$ with some cost in $\delta$.

\begin{lem}\label{disturbing the graphing plane}
Let $S\subset\Rb^n$ be a $k$-dimensional $(\delta,r)$-graphical submanifold.
\bnu
\item For each $x\in S$, let $l_x$ be a graphing plane. Then $d_H|_{B_{3r}(x)}(T_xS,l_x)\le C(n)\delta r$.
\item For each $x\in S$, for any $k$-plane $l_x'$ with $d_H|_{B_{r}(x)}(l_x,l'_x)\le \delta r$, there exists $g'_x:l'_x\to l'^\perp_x$ with $||g'_x||\le C(n)\delta r, ||\nabla g'_x||\le C(n)\delta, ||\nabla^2 g'_x||\le C(n)\delta r^{-1}$ such that $S\cap B_r(x)=Graph(g'_x)\cap B_r(x)$.
\enu
\end{lem}

\pf

\bnu

\item Let $g_x$ be the respective graphing function. Assume $x =(y, g_x(y))$. Let $\tilde{g}(y')=\nabla g_x(y'-y)+g_x(y)$. Then $T_xS = Graph(\tilde{g})$. Note that $||g_x||\le \delta, ||\nabla g_x||\le \delta r$. We have $|\tilde{g}|\le C(n)\delta r$ in $B_{3r}(x)$. Thus $d_H|_{B_{3r}(x)}(T_xS,l_x)\le C(n)\delta r$.

\item Let $l_x$, $g_x$ be the graphing plane and graphing function. For $w = (y,z)\in l_x\times l_x^\perp$, define $f(w) = z-g_x(y)$. Then $\nabla f = (-\nabla g_x , I_k)$, $\nabla^2f = (-\nabla^2g_x,0)$. Now we view $f$ as a function $l'_x\times l'^\perp_x \to l^\perp_x$. Since $d_H|_{B_{r}(x)}(l_x,l'_x)\le \delta r$, in $B_{3r}(x)$, we still have
\begin{equation}
    ||f||\le \delta r,\; ||\nabla f-(0, I_k)||\le C(n)\delta, \;||\nabla^2 f||\le C(n) \delta r^{-1}.
\end{equation}
By Proposition \ref{quantitative implicit function theorem}, there exists $g'_x$ such that $f(y',g'_x(y')) = 0$ for all $y' \in B_{2r}(x)\cap l'_{x}$, with $||g'_x||\le C(n)\delta r, ||\nabla g'_x||\le C(n)\delta, ||\nabla^2 g'_x||\le C(n)\delta r^{-1}$ in $B_{2r}(x)\cap l'_x$. 

It is clear that $S\cap B_r(x)=Graph(g'_x)\cap B_r(x)$. We may modify $g'_x$ outside $B_{r}(x)$ to make the estimates holds on the whole $l'_x$.
\enu

\epf

A corollary of Lemma \ref{disturbing the graphing plane} shows that the graphing plane moves slowly on $S$.
\begin{cor}

Let $S\subset\Rb^n$ be a $k$-dimensional $(\delta,r)$-graphical submanifold. For $x,y \in S$ with $d(x,y)<r$, we have $d(\hat{l}_x,\hat{l}_y)\le C(n)\delta$ and $d_H|_{B_r(x)}(l_x,l_y)\le C(n)\delta r$.
    
\end{cor}
\pf

By Lemma \ref{disturbing the graphing plane} and Lemma \ref{local Hausdorff distance properties}, $d_H|_{B_{3r}(x)}(l_x,T_xS)\le C(n)\delta r$, $d_H|_{B_{3r}(y)}(l_y,T_yS)\le C(n)\delta r$. It suffices to check $d_H|_{B_{2r}(x)}(T_xS,T_yS)\le C(n)\delta r$, which follows from the estimates on $g_x$. 

\epf

The following lemma is an application of the implicit function theorem, which describes the local level sets $\delta$-regular maps.

\begin{lem}\label{local level set}

For $B_{1.5r}(x_0)\subset \Rb^n$, let $f:B_{1.5r}(x_0)\to \Rb^k$ be a $\delta r^{-1}$-regular map with $\delta <\delta(n)$. Let $l_{x_0}=\ker \nabla f(x_0)+x_0$. Then there exists a smooth $g_{x_0}:l_{x_0}\to l_{x_0}^\perp$ with $r_{x_0}||g_{x_0}||$, $||\nabla g_{x_0}||$, $r_{x_0}^{-1}||\nabla^2g_{x_0}||\le C(n)\delta$, such that $f^{-1}(f(x_0))\cap B_r(x_0)=Graph(g_{x_0})\cap B_r(x_0)$.

\end{lem}

\pf

Note that $f'=l(\nabla f(x_0))^{-1}f$ is also $\delta r^{-1}$-regular and has the same level sets as $f$. We may assume $\nabla f(x_0)\in O_{k\times n}(\Rb)$. Lemma \ref{regular map equivalent definition} implies that $||\nabla^2 f||_{B_{2r}(x_0)}\le 2\delta r^{-1}$. By Lemma \ref{invariance of norm}, $|\nabla^2f|$ stays invariant under change of orthogonal basis. Thus without loss of generality we may assume $x_0=0^n$, $\nabla f(x_0)=(0,I_k)$ and $l_{x_0}=\{0^{n-k}\}\times \Rb^k$.

Note that $B_r(0^n)\subset B_r(0^{n-k})\times B_r(0^k)\subset B_{1.5r}(0^n)$. By Proposition \ref{quantitative implicit function theorem} and Remark \ref{continuous extension}, the conclusion follows immediately.

\epf

\pf[Proof of Proposition \ref{level sets of delta_x regular maps}]
By Lemma \ref{local level set}, for each $x_0\in f^{-1}(c)$, there exists $g_{x_0}:l_{x_0}\to l_{x_0}^\perp$, with $r_{x_0}||g_{x_0}||$, $||\nabla g_{x_0}||$, $r_{x_0}^{-1}||\nabla^2g_{x_0}||\le C(n)\delta$, such that $S\cap B_{r}(x_0)=Graph(g_{x_0})\cap B_{r}(x_0)$. Thus the conclusion follows.
\epf

Additionally, the techniques in this section describes the image of regular maps. We start from the following local lemma.

\begin{lem}\label{local image of regular map}
For $B_{r}(x_0)\subset \Rb^n$, let $f:B_{r}(x_0)\to \Rb^k$ be a $\delta r^{-1}$-regular map with $\delta <\delta(n)$. Let $r'=(1-C(n)\delta)r/|L(\nabla f(x_0))^{-1}|_{op}$. Then for any $y\in B_{r'}(f(x_0))$, there exists $x\in B_r(x_0)\cap f^{-1}(y)$ with 
\begin{equation}
    |x-x_0|\le (1+C(n)\delta)|L(\nabla f(x_0))^{-1}|\cdot |y-y_0|.
\end{equation}
\end{lem}

\pf

Let $f'=L(\nabla f(x_0))^{-1}f$. By change of coordinate on $\Rb^n$ we may assume $\nabla f'(x_0)=(0,I_k)$. Denote $x_0=(y_0,z_0)\in \Rb^{n-k}\times \Rb^k$. Define $g:B_r(v_0)\to \Rb^k:z\mapsto f'(y_0,z)$. Then $g(v_0)=f'(x_0)$, $||\nabla g-I_k||\le \delta$ in $B_r(v_0)$.

For $y'=L(\nabla f(x_0))^{-1}y$, $|y'-f'(x_0)|\le (1-C(n)\delta)r$. By Proposition \ref{quantitative inverse function theorem}, there exists $v\in B_r(v_0)$ such that $g(v) =y'$, or equivalently $f(u_0,v)=y$, with
\begin{equation}
    |v'-v_0|\le (1+C(n)\delta)|y'-y_0|\le (1+C(n)\delta)|L(\nabla f(x_0))^{-1}|\cdot |y-y_0|.
\end{equation}
Taking $x=(u_0,v)$ completes the proof.

\epf

Thus a regular map $f$ with a global $||L(\nabla f)^{-1}||$ bound
is actually surjective. And we may find the preimage of any point within a certain region.
\begin{prop}\label{distance between level sets of regular map}
    Let $f:\Rb^n\to \Rb^k$ be a $C$-regular map. Assume $||L(\nabla f)^{-1}||\le \Lambda$. Then $f$ is surjective. For any $c,d\in \Rb^k$, $d_H(f^{-1}(c),f^{-1}(d))\le \Lambda|c-d|$.
\end{prop}
\pf
Take $r$ small and let $\delta = Cr$. Then $f$ is $\delta r^{-1}$ regular. By Lemma \ref{local image of regular map}, for any $c,c_1\in \Rb^k$ with $|c-c_1|\le (1-C(n)\delta)r/\Lambda$, if $x\in f^{-1}(c)$, then there exists $x_1\in f^{-1}(c_1)$ with $|x-x'|\le (1+C(n)\delta)\Lambda |c-c_1|$.

Taking a sequence of points $c_0,c_1,c_2,\dots,c_N$ on the segment $cd$ with $c_0=c$, $c_N=d$, we may assume $|c_i-c_{i+1}|\le (1-C(n)\delta)r/\Lambda$. Thus we may pick $x_1,\dots,x_N$ with $|x_i-x_{i+1}|\le (1+C(n)\delta)\Lambda |c_i-c_{i+1}|$. Thus $f(x_N)=d$ and
\begin{equation}
    |x_N-x|\le (1+C(n)\delta)\Lambda|c-d|.
\end{equation}
Thus, $d_H(f^{-1}(c),f^{-1}(d))\le (1+C(n)\delta)\Lambda|c-d|$. Since $\delta\to 0$ as $r\to 0$, we have 
\begin{equation}
    d_H(f^{-1}(c),f^{-1}(d))\le \Lambda |c-d|.
\end{equation}
\epf

\subsection{Closest Point Projection to Graphical Submanifolds}

In this subsection, we will recall the closest point projection and continue the discussion about level sets of regular maps.

\begin{dfn}
    Let $A\subset \Rb^n$ be a closed subset, for any $x\in \Rb^n$, define the closest point projection $P_A(x)$ to be the nearest point on $A$ to $x$:
    \begin{equation}
        P_A(x):=\argmin{y\in A}d(x,y).
    \end{equation}
\end{dfn}

The closest point projection is a set in general, but we'll see that for a graphical submanifold $S$, $P_S$ is a single valued map near $S$.

\begin{prop}\label{projection to graphical submanifold}

Let $S\subset \Rb^n$ be a $k$-dimensional $(\delta,r)$-graphical submanifold with $\delta<\delta(n)$. Then the closest point projection map $P_S:B_{r/2}(S)\to S, x\mapsto \argmin{y\in S}d(x,y)$ is well defined and smooth, with $|\nabla P_S(y) - \hat{ P}_{l_x}| < C(n)\delta$ for any $y\in B_{r/2}(S)$.
    
\end{prop}

The closest point projection between the level sets of two maps nearby is a biLipschitz diffeomorphism.

\begin{prop}\label{projection between level sets is diffeomorphism}
For $\delta<\delta(n)$, $r>0$, assume $f,g:\Rb^n\to \Rb^k$ are $\delta r^{-1}$-regular maps, with $||L(\nabla f)^{-1}(f-g)||\le \delta r$, $||L(\nabla f)^{-1}(\nabla f-\nabla g)||\le \delta$. Then for any $c\in \Rb^k$, the closest point projection $P_{g^{-1}(c)}:f^{-1}(c)\to g^{-1}(c)$ is a diffeomorphism. Moreover,
\bnu
\item For any $x\in f^{-1}(c)$, $|P_{g^{-1}(c)}(x)-x|\le C(n)\delta r$.
\item $P_{g^{-1}(c)}$ is biLipschitz: for $x, y\in f^{-1}(c)$, 
\begin{equation}
    (1-C(n)\delta) |x-y|\le|P_{g^{-1}(c)}(x)-P_{g^{-1}(c)}(y)|\le (1+C(n)\delta) |x-y|.
\end{equation}
\enu
\end{prop}

To prove Proposition \ref{projection to graphical submanifold} and \ref{projection between level sets is diffeomorphism}, we first need a regularity control for closest point projection maps.

\begin{lem}\label{projection to a graph}

Let $f:B_1(0^{n-k})\to \Rb^k$ be a smooth function with $|f(0^{n-k})|$, $|\nabla f(0^{n-k})|$, $||\nabla^2 f||\le \delta$. Let 
$G=Graph(f)$. We may define the closest point projection map \begin{equation}
P_G:B_{0.99}(0^{n-k})\times B_{1}(0^k)\to G,\, x\mapsto \argmin{y\in G}d(x,y).
\end{equation}
Then $P_G$ is well-defined and smooth, with $|P_G(0^{n-k})|$, $||\nabla P_G-(I_{n-k},0)||\le C(n)\delta$
\end{lem}

\pf

For each $x=(x_1,x_2)\in B_{0.99}(0^{n-k})\times B_{1}(0^k)$, it is clear that there exists at least one nearest point on $G$ to $x$. If $(y,f(y))$ is a nearest point to $x$, then $|y-x_1|<0.01$ and $\nabla_y|x-(y,f(y))|^2=0$, where $\nabla_y$ yields a column vector.

Define $g: B_{0.99}(0^{n-k})\times B_{1}(0^k) \times B_{0.01}(0^{n-k})\to \Rb^{n-k}$,
\begin{equation}
    g(x,y)=\frac{1}{2}\nabla_y|x-(x_1+y,f(x_1+y))|^2.
\end{equation}
Then $\nabla_y|x-(y,f(y))|^2=0$ is equivalent to $g(x,y-x_1)=0$.
% To apply Proposition \ref{quantitative implicit function theorem} to $g$, we need to check the estimates for $g$. 
Compute
\begin{equation}
    g(x,y) = \frac{1}{2}\nabla_y|(-y, x_2-f(x_1+y))|^2 = y +\langle x_2-f(x_1+y), \nabla f(x_1+y)\rangle
\end{equation}
Since $||f||$, $||\nabla f||$, $||\nabla^2f||\le \delta$, we yield
\begin{equation}
    |g(0^n,0^{n-k})|,\; ||\nabla g-(0,I_{n-k})||\le C(n)\delta.
\end{equation}
Since $B_{0.99}(0^{n-k})\times B_{1}(0^k)$ is convex, by a slightly modified version of the $C^1$ version of Proposition \ref{quantitative implicit function theorem}, there exists a $u:B_{0.99}(0^{n-k})\times B_{1}(0^k)\to B_{0.01}(0^{n-k})$, such that $u(x)$ is the unique point with $g(x,u(x))=0$, with $||u||$, $||\nabla u||\le C(n)\delta$. By the uniqueness of $u(x)$, $(x_1+u(x),f(x_1+u(x)))$ must be the unqiue nearest point on $G$ to $x$. Thus $P_G$ is well-defined:
\begin{equation}
    P_G(x)=(x_1+u(x),f(x_1+u(x))).
\end{equation}
The estimates follows from the expression of $P_G$ immediately.
\epf

Lemma \ref{projection to a graph} easily extends to an entire graphical submanifold.

\pf[Proof of Proposition \ref{projection to graphical submanifold}]

For each $x\in S$, $y\in B_{r/2}(S)$, the nearest point to $y$ on $S$ lies in $S\cap B_{r}(x)$. Since $S$ is a graph in $B_{r}(x)$ with graphing plane $l_x$, the conclusion follows from Lemma \ref{projection to a graph}.

\epf

If we have two almost zero functions, then the closest point projection between their graphs is a local diffeomorphism. The following lemma is a local version of Proposition \ref{projection between level sets is diffeomorphism}.

\begin{lem}\label{projection between graphs}

Let $f,g:B_1(0^{n-k})\to \Rb^k$ be smooth functions with $|f(0^{n-k})|$, $|\nabla f(0^{n-k})|$, $||\nabla^2 f||$, $|g(0^{n-k})|$, $|\nabla g(0^{n-k})|$, $||\nabla^2 g||\le \delta$. Let $P: Graph(f|_{B_{0.99}(0^{n-k})})\to Graph(g)$, be the closest point projection map. Then $P$ is a diffeomorphism onto its image. Moreover, we have
\bnu
\item $|P(x)-x|\le C(n)\delta$.
\item $P$ is $(1+C(n)\delta)$ biLipschitz, i.e.
\begin{equation}
    (1-C(n)\delta)|x-y|\le|P(x)-P(y)|\le (1+C(n)\delta)|x-y|.
\end{equation}
\item $Graph(g|_{B_{0.9}(0^{n-k})})\subset P(Graph(f|_{B_{0.99}(0^{n-k})}))$.
\enu.
    
\end{lem}
\pf
Define 
\begin{equation}
    F:B_{0.99}(0^{n-k})\to \Rb^{n-k},\,x\to \pi_1(P((x,f(x))),
\end{equation} where $\pi_1$ takes the first $n-k$ coordinates. By the conditions and Lemma \ref{projection to a graph}, $F$ is well-defined and smooth, with $|F(0^{n-k})|,\,||\nabla F - I_{n-k}||\le C(n)\delta$.

By the $C^1$ version of Proposition \ref{quantitative inverse function theorem}, $F$ is a $(1+C(\delta))$-biLipschitz diffeomorphism onto its image, with $B_{0.9}(0^{n-k})\subset F(B_{0.99}(0^{n-k}))$. Thus $P$ is also a $(1+C(n)\delta)$-biLipschitz diffeomorphism onto its image, with $Graph(g|_{B_{0.9}(0^{n-k})})\subset P(Graph(f|_{B_{0.99}(0^{n-k})}))$.
\epf

Now we may prove Proposition \ref{projection between level sets is diffeomorphism}.

\pf[Proof of Proposition \ref{projection between level sets is diffeomorphism}]

Take $x_0\in f^{-1}(c)$. Since $||L(\nabla f)^{-1}(\nabla f-\nabla g)||\le \delta$, we have 
\begin{equation}
    |L(\nabla f)^{-1}L(\nabla g)-I_k|\le C(n)\delta.
\end{equation}
Similar to the proof of Lemma \ref{local level set}, we may assume $L(\nabla f(x_0))=I_k$, $\nabla f(x_0) = (0, I_k)$, $x_0=0^n$, $c = 0^k$, without loss of generality. Under this setting, by Lemma \ref{regular map equivalent definition}, the conditions imply that
\begin{equation}
\begin{split}   
    &f(0^n) = 0,\; \nabla f(0^n) = (0, I_k),\; ||\nabla^2 f||_{L^\infty(B_{1.5r}(0^n))}\le C(n)\delta r^{-1},\\
    |g(0^n&)|\le \delta r,\; |\nabla g(0^n)-(0, I_k)|\le \delta ,\; ||\nabla^2 g||_{L^\infty(B_{1.5r}(0^n))}\le C(n)\delta r^{-1}.
\end{split}
\end{equation}
Apply Proposition \ref{quantitative implicit function theorem} for $f$ and $g$. There exits $u,v:B_{r}(0^{n-k})\to \Rb^k$, with $r^{-1}||u||$, $||\nabla u||$, $r||\nabla^2u||$, $r^{-1}||v||$, $||\nabla v||$, $r||\nabla^2v||\le C(n)\delta$ in $B_{r}(0^n)$, such that 
\begin{equation}
\begin{split}    
    f^{-1}(0^k)\cap (B_{r}(0^{n-k})\times B_{r}(0^k))=Graph(u),\\
    g^{-1}(0^k)\cap (B_{r}(0^{n-k})\times B_{r}(0^k))=Graph(v). 
\end{split}
\end{equation}
By Lemma \ref{projection between graphs}, we have
\bnu
\item $P_{g^{-1}(c)}$ is a $(1+C(n)\delta)$-biLipschitz diffeomorphism from $f^{-1}(c)\cap B_{0.99r}(x_0)$ onto its image.
\item $|P_{g^{-1}(c)}(x_0)-x_0|\le C(n)\delta r$.
\item $g^{-1}(c)\cap B_{0.9r}(x_0)\subset P_{g^{-1}(c)}(f^{-1}(c)\cap B_{0.99r}(x_0)).$
\enu
Since $x_0$ is arbitrary, we've proved that $P_{g^{-1}(c)}$ is a local diffeomorphism with $|P_{g^{-1}(c)}(x)-x|\le C(n)\delta r_x$. It remains to check that $P_{g^{-1}(c)}$ is a globally biLipschitz diffeomorphism. 

For $x,y\in g^{-1}(c)$, if $y\in B_{0.99r}(x)$, then we've proved that 
\begin{equation}\label{biLipschitz of P_g^-1}
    (1-C(n)\delta) |x-y|\le|P_{g^{-1}(c)}(x)-P_{g^{-1}(c)}(y)|\le (1+C(n)\delta) |x-y|.
\end{equation}
Otherwise, $|P_{g^{-1}(c)}(x)-x|\le C(n)\delta r$, $|P_{g^{-1}(c)}(y)-y|\le C(n)\delta r$, $|x-y|\ge 0.99 r$. Sum these up, and we also get (\ref{biLipschitz of P_g^-1}), i.e. $P_{g^{-1}(c)}$ is $(1+C(n)\delta)$-biLipschitz onto its image.

(\ref{biLipschitz of P_g^-1}) implies that $P_{g^{-1}(c)}$ is injective. For any $y\in g^{-1}(c)$, let $x=P_{f^{-1}(c)}(y)$. Then $|x-y|\le C(n)\delta r$. Thus
\begin{equation}
    y\in g^{-1}(c)\cap B_{0.9 r}(x)\subset P_{g^{-1}(c)}(f^{-1}(c)\cap B_{0.99r}(x)).
\end{equation}
Thus $P_{g^{-1}(c)}$ is also surjective. Since $P_{g^{-1}(c)}$ is a bijective local diffeomorphism, it is a diffeomorphism.
\epf

\section{Splitting Reifenberg Condition}

In this short section, we'll introduce the concept of $(k,\delta)$-splitting and $(k,\delta,N)$-splitting Reifenberg sets. For a splitting Reifenberg set $S$ with small $\delta$, the `bad scales' where $S$ resembles multiple planes are quite sparse.

\subsection{Almost Splitting}

In this subsection, we will introduce almost splitting, almost splitting sets and almost splitting directions. We'll see that the selection of almost splitting sets and directions is almost unique if the almost splitting set is finite parallel planes.

First we recall the definition of the exact splitting.

\begin{dfn}

We call a closed set $A\subset \Rb^n$ a $k$-splitting if there exists a $k$-dimensional linear subspace $\hat{l}$ with $\hat{l}+A=A$. We say $\hat{l}$ is a splitting direction of $A$.

\end{dfn}

\begin{rem}

The splitting direction is unique if $A$ is not a $k+1$ splitting, which will always be true in this notes.

\end{rem}

A set $S$ is almost splitting in a ball if it is locally Hausdorff close to an exact splitting space.

\begin{dfn}[$(k,\delta)$-splitting]\label{almost splitting}

We call $S\subset \Rb^n$ a $(k,\delta)$-splitting in $B_r(x)\subset \Rb^n$ if there exists a $k$-splitting $A_{x,r}\subset \Rb^n$ with splitting direction $\hat{l}_{x,r}$ such that
\begin{equation}
    d_H|_{B_r(x)}(S,A_{x,r})\le \delta r.
\end{equation}
Here we call $A_{x,r}$ the $(k,\delta)$-splitting set and $\hat{l}_{x,r}$ the $(k,\delta)$-splitting direction of $S$ in $B_r(x)$.

\end{dfn}

Let's give an example of $(k,\delta)$-splitting and $(k,\delta)$-splitting set.

\begin{figure}
    \centering
    \includegraphics[width=0.5\linewidth]{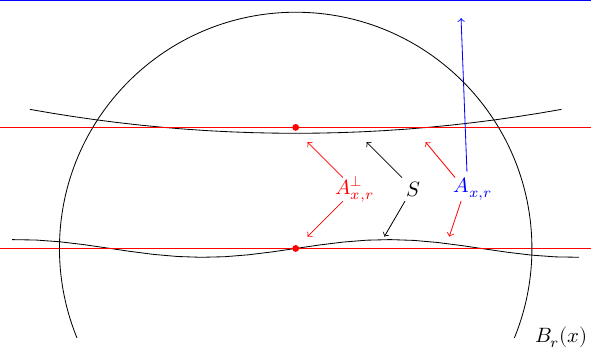}
    \caption{Example \ref{unreduced splitting set} and Lemma \ref{reducing the splitting set}}
    \label{fig3}
\end{figure}
\begin{expl}\label{unreduced splitting set}
    (Figure \ref{fig3}) Take the ball $B_1(0^2)\subset \Rb^2$. Let $f_1,f_2:\Rb\to \Rb$ be two functions with $||f_1||\le \delta$, $||f_2-0.5||\le \delta$. Let $S=Graph(f_1)\cup Graph(f_2)$, $A_{0,1}:=\Rb\times\{0, 0.5, 1.1\}$.
\end{expl}
\begin{proof}[Explanation:] It is clear that $S$ is 
a $(1,\delta)$-splitting in $B_1(0^2)$ with splitting set $A_{0,1}$ and splitting direction $\hat{l}_{0,1}=\Rb\times\{0\}$, i.e., $d_H|_{B_1(0^2)}(S,A_{0,1})\le \delta$. Moreover, $A_{0,1}=\hat{l}+\{(0,0), (0,0.5), (0,1.1)\}$.

We want to measure the size of $S$ in $B_r(x)$. It is clear that $S$ resembles two lines, instead of three. The red lines $\hat{l}+(0,0)$ and $\hat{l}+(0,0.5)$ are related to the size of $S$ since they approximate $S$, while the blue line $\hat{l}+(0,1.1)$ is actually irrelevant to $S$. $A_{0,1}$ remains a $(1,\delta)$-splitting set of $S$ in $B_1(0^2)$ after removing that blue line. Since the distance between the red lines is $0.5$, we can say that the size of $S$ in $B_1(0^2)$ is approximately $0.5$ in some sense.
\end{proof}

From Example \ref{unreduced splitting set} we see that the size of $S$ in $B_r(x)$ can be measured by $A_{0,1}$, but only by its relevant part. We have the following lemma to summarize this phenomenon.

\begin{lem}\label{reducing the splitting set}

(Figure \ref{fig3}) For $S\subset \Rb^n$, $B_r(x)\subset \Rb^n$, assume $A_{x,r}$ is a $(k,\delta)$-splitting set of $S$ in $B_r(x)$ with splitting direction $\hat{l}_{x,r}$. Let $l^\perp_{x,r}=\hat{l}_{x,r}^\perp+x$. Take 
\begin{equation}
A_{x,r}^\perp=\{a\in A_{x,r}\cap l_{x,r}^\perp\,|\,dist(\hat{l}_{x,r}+a,S\cap B_r(x))\le \delta r\}.
\end{equation}
and $\tilde{A}_{x,r}:=\hat{l}_{x,r}+A_{x,r}^\perp$. Then $\tilde{A}_{x,r}$ is also a splitting set of $S$ in $B_r(x)$ with splitting direction $\hat{l}_{x,r}$. Moreover,
\begin{equation} 
A_{x,r}\cap l_{x,r}^\perp\cap B_{(1-\delta)r}(x)\subset A_{x,r}^\perp \subset A_{x,r}\cap l_{x,r}^\perp\cap\overline{B}_{(1+\delta)r}(x).
\end{equation}

\end{lem}

\pf

For any $y\in S\cap B_r(x)$, since $A_{x,r}$ is closed, there is $a\in A_{x,r}^\perp$ with $d(y,\hat{l}+a)\le \delta r$. By definition $a\in A_{x,r}^\perp$. Thus $d_H|_{B_r(x)}(S,\tilde{A}_{x,r})\le \delta r$ and $\tilde{A}_{x,r}$ is also a $(k,\delta)$-splitting set of $S$ in $B_r(x)$.

For $a\in A_{x,r}\cap l_{x,r}^\perp\cap B_{(1-\delta)r}(x)$, $d(a,S)\le \delta r$. Hence $a\in A_{x,r}^\perp$. For $a\in A_{x,r}^\perp$, 
\begin{equation}
    d(a,B_r(x))= dist(\hat{l}+a,B_r(x))\le dist(\hat{l}+a,S\cap B_r(x))\le \delta r.
\end{equation}
Thus $A_{x,r}^\perp\subset A_{x,r}\cap l_{x,r}^\perp\cap \overline{B}_{(1+\delta)r}(x)$.
\epf

\begin{rem}\label{notation of splitting direction and set}

In this article we'll always use the notations in the above discussion.

\bnu 

\item$A_{x,r}$ and $\hat{l}_{x,r}$ denote a pair of splitting set and splitting direction of $S$ in $B_r(x)$, with $A_{x,r}=\hat{l}_{x,r}+A_{x,r}$. The selection is not unique, and we will simply choose an arbitrary pair.

\item We use $l_{x,r}^\perp$ to denote the $(n-k)$-plane that is orthonormal to $\hat{l}_{x,r}$ and passes $x$, and
\begin{equation}
    A_{x,r}^\perp=\{a\in A_{x,r}\cap l_{x,r}^\perp\,|\,d(\hat{l}_{x,r}+a,S\cap B_r(x))\le \delta r\}.
\end{equation}

\enu
    
\end{rem}

% \begin{cor}\label{small diameter means reduced}

% Assume the same as in Lemma \ref{reducing the splitting set}. If $S\cap B_{r/2}(x)\neq \emptyset$ and $\diam{A^\perp}\le (1/2-2\delta)r$, then $A$ is a $(k,\delta)$-splitting set of $S$ in $B_r(x)$.

% \end{cor}

% \pf

% For $y\in S\cap B_{r/2}(x)$, $d(y,\hat{l}+a)\le \delta r$ for some $a\in A^\perp\cap B_{(1/2+\delta)r}(x)$. Thus $A^\perp\subset B_{(1-\delta)r}(x)$. By Lemma \ref{reducing the splitting set}, $A^\perp = A^\perp_{x,r}$ is a $(k,\delta)$-splitting set.

% \epf

It is clear that $\hat{l}_{x,r}$ and $A^\perp_{x,r}$ are not unique. This is inevitable due to the nonrestrictive nature of the definition, but at least we hope them to be almost unique, i.e. all potential choices of $A^\perp_{x,r}$ and $\hat{l}_{x,r}$ should be $C\delta$ close to each other. This can clearly fail when 
$S$ is $(k+1)$-splitting for example. Fortunately, this is true if $A_{x,r}^\perp$ is finite.

\begin{lem}\label{splitting direction is almost unique}

Let $S$ be a closed subset of $\Rb^n$ with $S\cap B_{r/2}(x)\neq \emptyset$. Assume $S$ is a $(k,\delta)$-splitting in $B_r(x)$ with splitting sets $A_{x,r}$, $A'_{x,r}$ and splitting directions $\hat{l}_{x,r}$, $\hat{l}'_{x,r}$, where $|A^\perp_{x,r}|\le N$. Then if $\delta\le \delta(n,N)$, we have $d(\hat{l}_{x,r},\hat{l}'_{x,r})\le 5\delta$, $d_H(A_{x,r}^\perp,A'^\perp_{x,r})\le 7\delta r$.

\end{lem}

\pf
Without loss of generality, we assume $B_r(x)=B_1(0^n)$. Denote $A_{x,r}, A'_{x,r},A_{x,r}^\perp, A_{x,r}'^\perp,\hat{l}_{x,r},\hat{l}'_{x,r}$ as $A,A',A^\perp, A'^\perp,\hat{l},\hat{l}'$ respectively.

\textbf{Step 1:} We prove $d(\hat{l},\hat{l}')\le 5\delta$.

From the conditions we know $d_H|_{B_{1-\delta}(0^n)}(A,A')\le 2\delta$. Denote $d(\hat{l},\hat{l}')$ as $d$. First we give a weaker bound of $d$. Note that for any $k$-planes $l\subset A$, $l'\subset A'$. \begin{equation}
\Hc^k(l\cap B_1(0^n) \cap B_{2\delta}(l'))\le C(n)d^{-1}\delta.\end{equation} 
Fix some $l'\subset A'$, 
\begin{equation}
l'\cap B_{1-\delta}(0^n)\subset B_{2\delta}(A).
\end{equation} 
Since $S\cap B_{1/2}(0^n)\neq \emptyset$, $\Hc^k(l'\cap B_{1-\delta}(0^n))\ge C(n)$. Thus $NC(n)d^{-1}\delta \ge C(n)$, i.e. 
\begin{equation}\label{first bound of d of splitting directions}
    d\le C(n)N\delta.
\end{equation}

\begin{figure}\label{fig4}
\centering
\includegraphics[]{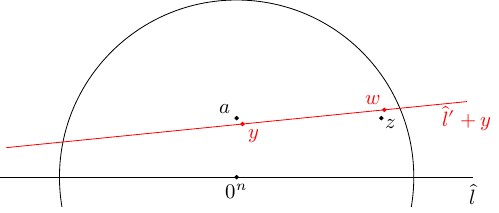}
\caption{proof of Lemma \ref{splitting direction is almost unique}, Step 1. $z_{k+1}\le a_{k+1}$.}
\end{figure}

Next we try to refine the estimate. (Figure \ref{fig4}) 

By Lemma \ref{equivalent definition of distance of subspaces}, there exist $v\in \hat{l}'\cap \bar{B}_1(0^n)$ with $d(v,\hat{l})=d$. We may pick a coordinate system such that $\hat{l}=\Rb^k\times\{0^{n-k}\}$ and $v=(\sqrt{1-d^2},0,\dots,0,d,0,\dots,0)$, where the $k+1$-th coordinate $v_{k+1}=d$. Without loss of generality we may assume $A^\perp\subset \{0^k\}\times\Rb^{n-k}$. For $\delta\le \delta(n,N)$ small enough, note that $|A^\perp|\le N$ and $S\cap B_{1/2}(0^n)\neq \emptyset$, we can find some $a\in A^\perp$ satisfying

\bnu

\item $|a|\le 0.51$,

\item for any $b\in A^\perp$ with $d(a,b)\le 10C(n)N\delta$, where $C(n)$ is from (\ref{first bound of d of splitting directions}), we have $a_{k+1}\ge b_{k+1}$.

\enu

Since $a\in A^\perp\cap B_{1-\delta}(0^n)$, there is some $y\in A'$ with $d(y,a)\le 2\delta$. Then $w:=y+0.8v\in A\cap B_{1-\delta}(0^n)$. Again, there exist $z\in \hat{l}+b\subset A$ with $d(w,z)\le 2\delta$ for some $b\in A^\perp$. Note that $d(a,b)\le 10 C(n)N\delta$, and 
\begin{equation}
b_{k+1}= z_{k+1}\ge w_{k+1}-2\delta=y_{k+1}+0.8d-2\delta\ge a_{k+1}+0.8d-4\delta.
\end{equation}
Since $a_{k+1}\le b_{k+1}$, it follows that $d\le 5\delta$.

\textbf{Step 2:} We prove $d_H(A^\perp,A'^\perp)\le 7\delta$.

Take $a\in A^\perp$. Then $d(S\cap B_r(x),\hat{l}+a)\le \delta$. Take $z\in S\cap B_r(x)$ with $d(z,\hat{l}+a)\le \delta + \delta_1$ where $\delta_1<<\delta$. There exists $a'\in A'^\perp$ such that $d(z,\hat{l}'+a')\le \delta$. Let $P^\perp$, $P'^\perp$ be the orthogonal projection maps from $\Rb^n$ to $\hat{l}^\perp$, $\hat{l}'^\perp$. By 1. and Lemma \ref{equivalent definition of distance of subspaces}, 
\begin{equation}d(\pi^\perp(z),\pi'^\perp(z))\le ||P^\perp-P'^\perp||_{op}|z|\le d(\hat{l},\hat{l}')\le 5\delta.\end{equation}
Note that $d(P^\perp(z),a)\le \delta + \delta_1$, $d(P'^\perp(z),a')\le \delta$. We have $d(a,a')\le 7\delta +\delta_1$. Let $\delta_1\to 0$, we have $d_H(A^\perp,A'^\perp)\le 7\delta$.

\epf

\subsection{$(k,\delta,N)$-Splitting Reifenberg Condition}

In this article, we are concerned with the sets that satisfy the 
$(k,\delta,N)$-splitting Reifenberg condition, which is an analog of the $(k,\delta)$-Reifenberg condition.

\begin{dfn}[splitting Reifenberg condition]\label{splitting reifenberg condition}

For a closed set $S\subset B_2(0^n)$, we say $S$ satisfies the $(k,\delta,N)$-splitting Reifenberg condition if the following holds.
\bnu

\item $d_H|_{B_2(0^n)}(S,\Rb^k\times\{0^{n-k}\})\le 2\delta.$
\item For each $B_r(x)\subset B_2(0^n)$, $S$ is a $(k,\delta)$-splitting with splitting set $A_{x,r}$, such that $|A_{x,r}^\perp|\le N$.
\enu
\end{dfn}

\begin{rem}

Assume $S\subset B_2(0^n)$ is a $(k,\delta,N)$-splitting Reifenberg set. For $B_r(x)\subset B_2(0^n)$, as in Remark \ref{notation of splitting direction and set}, we will always use $\hat{l}_{x,r}$, $A_{x,r}$ to represent a pair of splitting direction and splitting set. Moreover, we will always assume $|A_{x,r}^\perp|\le N$.
    
\end{rem}

For a sequence of nested balls with rapidly decreasing scales, we can observe multiple sheets of $S$ at no more than $N$ of them before passing the scale $\delta$. We call these scales `bad scales' informally.

\begin{prop}\label{only finite scales have large width}

Let $S$ be a $(k,\delta,N)$-splitting Reifenberg set. Let $r_0>0$, $r_i=2^{-mi}r_0$, $i = 0, 1, \dots,M$ be a sequence of scales for some fixed $m\in \Nb$ . Take a sequence of nested balls $B_2(0^n)\supset B_{r_0}(x_0)\supset B_{r_1}(x_1)\supset\dots\supset B_{r_M}(x_M)$ with $S\cap B_{r_i/2}(x_i)\neq \emptyset$ for each $i$. Assume $\delta<\delta(n,N,m,M)$ is sufficiently small. Then there are at most $N$ different $0\le i\le M$ with
\begin{equation}\diam{A^\perp_{x_i,r_i}}\ge 2^{-m+2} r_i.\end{equation}

\end{prop}

\pf

There exists a $k$-splitting set $A=A_{x_0,r_0}$ with $|A|\le N$ such that
\begin{equation}
    d_H|_{B_r(x_0)}(S,A)\le \delta r_0.
\end{equation}
Thus $d_H|_{B_{r_i}(x_i)}(S,A)\le \delta r_0 = 2^{mi}\delta r_i$ for each $i$, i.e., $A$ is a $(k,2^{mi}\delta)$-splitting set for $S$ in $B_{r_i}(x_i)$. From the construction of Lemma \ref{reducing the splitting set}, we know $A_{x_0,r_0}^\perp\supset A_{x_1,r_1}^\perp\supset\dots\supset A_{x_M,r_M}^\perp$. For each $i$, if $\diam{A_{x_i,r_i}^\perp}\ge 3\cdot 2^{-m}r_i$, then $\diam{A_{x_i,r_i}^\perp}>\diam{A_{x_{i+1},r_{i+1}}^\perp}$, thus $|A_{x_{i+1},r_{i+1}}^\perp|<|A_{x_i,r_i}^\perp|$. Since $|A_{x_0,r_0}^\perp|\le N$, there are at most $N$ different $i$ with $\diam{A_{x_i,r_i}^\perp}\ge 3\cdot 2^{-m} r_i$.

For each $i\le M$, let $A'_{x_i,r_i}$ be the $(k,\delta)$-splitting set of $S$ in $B_{r_i}(x_i)$. Then it is also a $(k,2^{mi})$-splitting set. By lemma \ref{splitting direction is almost unique}, $|\diam{A^\perp_{x_i,r_i}}-\diam{A'^\perp_{x_i,r_i}}|\le 7\cdot2^{mM}\delta r_i<2^{-m} r_i$ for $\delta<\delta(n,N,m,M)$. Thus there are at most $N$ different $i\le M$ such that $\diam{A'^\perp_{x_i,r_i}}\ge 2^{-m+2}r_i$.

\epf

The splitting direction $\hat{l}_{x,r}$ and the splitting set $A^\perp_{x,r}$ is stable as we move the ball $B_r(x)$ or change the scale $r$. By Lemma \ref{splitting direction is almost unique}, the error is comparable to $\delta$.

\begin{cor}\label{comparison of splitting directions}

Let S be a $(k,\delta,N)$-Reifenberg set. Let $B_{r}(x)$, $B_{r'}(x')\subset B_2(0^n)$ be two balls with $S\cap B_{r/2}(x)\neq \emptyset$, $S\cap B_{r'/2}(x')\neq \emptyset$. Assume $\delta<\delta(n,N)$ is sufficiently small.

\bnu

\item If $B_{r'}(x')\subset B_r(x)$, $r'=r/2$, then $d(\hat{l}_{x,r},\hat{l}_{x',r'})\le 10\delta$.

\item If $B_{r'}(x')\subset B_r(x)$, $r'= 2^{-i}r$ and $B_{2r}(x)\subset B_2(0^n)$, then $d(\hat{l}_{x,r},\hat{l}_{x',r'})\le 10(i+2)\delta$.

\item If $B_r(x)\cap B_r(x')\neq \emptyset$ and $B_{4r}(x)\subset B_2(0^n)$, then $d(\hat{l}_{x,r},\hat{l}_{x',r'})\le 40\delta$.

\enu

\end{cor}

\pf

\bnu

\item

Let $A_{x,r}$, $A_{x',r'}$ be the respective splitting sets. By Lemma \ref{local Hausdorff distance properties}, $d_H|_{B_{r'}(x')}(S,A_{x,r})\le \delta r=2\delta (r/2)$. Thus $\hat{l}_{x,r}$ is a $(k,2\delta)$-splitting direction of $S$ in $B_{r'}(x')$ Apply Lemma \ref{splitting direction is almost unique} one sees $d(\hat{l}_{x,r},\hat{l}_{x',r'})\le 10\delta$.

\item Note that $d$ for subspaces is a metric. Apply 1. repeatedly for the sequence of balls $B_{2^{-i}r}(x')\subset B_{2^{-i+1}r}(x')\subset\dots\subset B_r(x')$, and we have $d(\hat{l}_{x',r'},\hat{l}_{x',r})\le 10i\delta$. Note that $B_{r}(x')\subset B_{2r}(x)$, $B_r(x)\subset B_{2r}(x)$. We have $d(\hat{l}_{x,r},\hat{l}_{x,2r})\le 10\delta$, $d(\hat{l}_{x',r},\hat{l}_{x,2r})\le 10\delta$. Combining all above, we have $d(\hat{l}_{x,r},\hat{l}_{x',r'})\le 10(i+2)\delta$.

\item Similar to above, consider the inclusion chains $B_r(x)\subset B_{2r}(x)\subset B_{4r}(x)$ and $B_{r}(x')\subset B_{2r}(x')\subset B_{4r}(x)$. Apply 1. repeatedly and the conclusion follows.

\enu

\epf

Apparently Corollary \ref{comparison of splitting directions} does not include all possible cases in practice, and there are many other settings in which we can compare the splitting directions. In general, for any two balls, as long as their radii and the distance between them are comparable, their splitting directions will also be comparable.

\section{The Global Map $\Phi$}

In this section we will prove the main theorem. And in the same process we will also see many interesting constructions and results. Given a splitting Reifenberg set $S$, we will construct a map $\Phi:\Rb^n\to \Rb^k$ such that the level sets of $\Phi$ is almost perpendicular to the splitting directions of $S$ when observed at each scale. Such $\Phi$ will be defined as the limit of a sequence of smooth $\Phi_i$, for which we can prove effective $C^2$ control. Consequently, $\Phi$ and level sets of $\Phi$ will both have H\"older structures. We will see that these results imply the main theorem due to fact that $S$ is closed and locally $(k,\delta)$-splitting.

\subsection{Construction of $\Phi_i$}\label{construction of phi_i}

In this section, we assume $S\subset B_2(0^n)$ is a $(k,\delta,N)$-splitting Reifenberg set. Let $r_i=2^{-mi}$ for some large integer $m$ that will be specifies later. We will construct $\Phi_i:\Rb^n\to \Rb^k$ such that the level sets of $\Phi_i$ is almost perpendicular to the splitting direction at scale $r_i$ near $S$.

This will be an inductive process. Without loss of generality, we may assume the splitting direction of $S$ in $B_2(0^n)$ is $\Rb^k\times\{0^{n-k}\}$. let $\Phi_0(x,y)=x$ be the coordinate projection and we always use $\Phi_{i}$ to construct $\Phi_{i+1}$. Take a good cover of $S\cap B_{1.99}(0^n)$ with balls of size $r_{i+1}$. We will first take the average of linear approximations of $\Phi_i$ in these balls, and then use another cutoff function to glue the new part with the background $\Phi_{i}$.

\begin{cons}
    A canonical cutoff function for all balls.
\end{cons}

Let $\phi:[0,1]\to [0,1]$ be a $C^\infty$ function with 
$\phi=1$ in $[0, 0.4)$, decreasing in $[0.4, 0.48]$ and $\phi=0$ in $(0.48,1]$.
We may assume $|\nabla \phi|, |\nabla^2 \phi|\le C$. For any ball $B_r(x)\subset \Rb^n$, define \begin{equation}\phi_{x,r}:=\phi(\frac{|x-r|}{r}).\end{equation} Then $\phi_{x,r}\in C_c^\infty(B_r(x))$, $|\nabla \phi|\le C(n)r^{-1}$, $|\nabla^2 \phi|\le C(n)r^{-2}$.

\begin{cons}\label{covering for all scales}

A covering for all scales.

\end{cons}

Let $S\subset \Rb^n$ be a locally $(k,\delta,N)$-splitting Reifenberg set. Let $r_i=2^{-mi}$, $i=1,2,\dots$ be a collection of scales, where $m$ is some integer to be decided later. For each $r_i$, we pick a collection of balls $\{B_{r_i}(x_{i,a})\}$ such that
\bnu
\item $x_{i,a}\in S\cap B_{1.99}(0^n)$,

\item $S\cap B_{1.99}(0^n)\subset \bigcup_{a}B_{0.01r_i}(x_{i,a})$,
\item $B_{0.002r_i}(x_{i,a})$ are mutually disjoint.
\enu

Then it is straightforward to see

\bnu
\item $B_{0.39r_i}(S\cap B_{1.99}(0^n))\subset \bigcup_a B_{0.4r_i}(x_{i,a})$.
\item The intersection number of the collection of balls $B_{r_i}(x_{i,a})$ is bounded above by $C(n)$.
\item $\bigcup_aB_{r_{i+1}}(x_{i+1,a})\subset\bigcup_a B_{0.02r_i}(x_{i,a})$.
\enu

\begin{cons}
    A partition of unity.
\end{cons}

For each $i,a$, let $\phi_{i,a}=\phi_{x_{i,a},r_i}$ be the respective cutoff function. Grant each $x_{i,a}$ the weight $\phi_{i,a}$. We may define a partition of unity:

\begin{equation}
\psi_{i,a}=\frac{\phi_{i,a}}{\sum_a\phi_{i,a}}.
\end{equation} Then
\bnu

\item For each $x\in \Rb^n$, there are at most $C(n)$ nonzero $\phi_{i,a}$ in a neighborhood of $x$.

\item In $\bigcup_aB_{0.4r_i}(x_{i,a})$, $\sum_a \phi_{i,a}(x)\ge 1$. $\psi_{i,a}$ is well defined and satisfies
\begin{equation}
|\nabla \psi_{i,a}|\le C(n)r_i^{-1},\quad|\nabla^2 \psi_{i,a}|\le C(n)r_i^{-2}.   
\end{equation}

\enu

We will also need a weight function to glue the $\Phi_{i+1}$ with the background $\Phi_i$. For each $i$, let $\phi'_{i,a}=\phi_{x_{i,a}, 0.5r_i}$. Take

\begin{equation}\chi_i=\frac{\sum_a\phi'_{i,a}}{\sum_a\phi'_{i,a}+\prod_a(1-\phi'_{i,a})}.\end{equation} 

Then
\bnu

\item $\sum_a\phi'_{i,a}+\prod_a(1-\phi'_{i,a})\ge 1$, and $\chi_i$ is well defined in $\Rb^n$.

\item For each $x$, at most $C(n)$ different $\phi'_{i,a} > 0$ in a neighborhood of $x$. Hence 
\begin{equation}
    |\nabla \chi_i|\le C(n)r_i^{-1},\quad |\nabla^2 \chi_i|\le C(n)r_i^{-2}.
\end{equation}

\item For each $x\in \bigcup_aB_{0.2r_{i+1}}(x_{i+1,a})$, $\phi'_{i,a}(x)=1$ for some $a$ and $\chi_i(x)=1$.

\item For each $x \not \in \bigcup_aB_{0.24r_{i+1}}(x_{i+1,a})$, $\phi'_{i,a}(x)=0$ for each $i$ and $\chi_i(x) = 0$.

\enu

\begin{cons}
    Construction of $\Phi_i$.
\end{cons}

For each $i$, we'll construct an approximate map $\Phi_i:\Rb^n\to \Rb^k$ to $\Phi$ above scale $r_i$. This can be done inductively. Without loss of generality, we may assume the splitting direction of $S$ in $B_2(0^n)$ is $\Rb^k\times\{0^{n-k}\}$. We start from the coordinate projection:
\begin{equation}
    \Phi_0:\Rb^{k}\times \Rb^{n-k}\to \Rb^k,\,(x,y)\mapsto x.
\end{equation}

Assume $\Phi_{i}$ has been selected. We first take the weighted average of some linear approximations of $\Phi_i$ for each $B_{r_{i+1}}(x_{i+1,a})$. Then glue it to the background $\Phi_{i}$ to get $\Phi_{i+1}$.

For each $x_{i+1,a}$, since $S$ is a $(k,\delta)$-splitting in $B_{r_{i+1}}(x_{i+1,a})$, we have a splitting direction $\hat {l}_{x_{i+1,a},r_{i+1}}$ and a corresponding orthogonal projection $\hat{P}_{{i+1},a}:\Rb^n\to \hat{l}_{x_{i+1,a},r_{i+1}}$. Take $f_{i+1,a}$ to be the affine map such that

\bnu

\item $f_{i+1,a}(x_{i+1,a})=\Phi_{i}(x_{i+1,a})$.

\item $\hat{f}_{i+1,a} = \nabla f_{i+1,a} := \nabla \Phi_{i}(x_{i+1,a})\circ \hat{P}_{i+1,a}$.

\enu

In $\bigcup_aB_{0.4r_{i+1}}(x_{i+1,a})$, we can take the average of $f_{i+1,a}$ by partition of unity and take

\begin{equation}\label{definition of Psi_i}
    \Psi_{i+1}=\sum_a\psi_{i+1,a}f_{i+1,a}.
\end{equation}

Then glue $\Psi_{i+1}$ with $\Phi_{i}$:

\begin{equation}\label{definition of Phi_i}
    \Phi_{i+1} = \chi_{i+1}\Psi_{i+1} + (1-\chi_{i+1})\Phi_{i}.
\end{equation}

\begin{rem}
$\Phi_i$ is defined on $\Rb^n$, although it is the trivial projection outside $B_2(0^n)$.
\end{rem}

\begin{rem}

$\Phi_{i+1}=\Phi_i$ outside $\bigcup_aB_{0.24r_{i+1}}(x_{i+1,a})$. $\Phi_{i+1}=\Psi_{i+1}$ in $\bigcup_aB_{0.2r_{i+1}}(x_{i+1,a})$. Note that 
\begin{equation}
    \bigcup_aB_{r_{i+1}}(x_{i+1,a})\subset \bigcup_aB_{0.02r_{i}}(x_{i,a}).
\end{equation}
$\Phi_{i+1}$ only modifies $\Phi_i$ on the region where $\Phi_i = \Psi_i$.
    
\end{rem}

\subsection{$C^0$, $C^1$ and $C^2$ Estimates of $\Phi_i$}

In this section we'll give the $C^0$, $C^1$ and $C^2$ controls of $\Phi_i$. We'll see that the level sets of $\Phi_i$ are almost perpendicular to $\hat{l}$ near $S$ at scale $r_i$, and $\nabla^2\Phi_i$ is small modulo $L(\nabla \Phi_i)$. The main result is the following lemma, which is stated under the setting of Section \ref{construction of phi_i}.

\begin{lem}\label{C^1 and C^2 modulo distortion}

Let $S\subset B_2(0^n)$ be a $(k,\delta,N)$-Reifenberg set. Assume $r_i=2^{-mi}$ be the scales and $\Phi_i$ be the maps defined in (\ref{definition of Phi_i}). Assume $m>C(n)$ is sufficiently large and $\delta<\delta(n,N,m)$ sufficiently small. Then there is some $C(n,N,m)$ such that the following hold for all $\Phi_i$.

\bnu

\item $\Phi_i$ is $C(n,N,m)\delta r_i^{-1}$-regular.
\item For $x\in \bigcup_a B_{0.4r_i}(x_{i,a})$, 
\begin{equation}
d((\ker\nabla \Phi_i(x))^\perp,\hat{l}_{x,r_i})\le C(n,N,m)\delta.
\end{equation}
\item For any $x\in \Rb^n$,
\begin{equation}
\begin{gathered}
||L(\nabla \Phi_i)(\Phi_{i+1}-\Phi_i)||\le C(n,N,m)\delta r_{i+1},\\
||L(\nabla \Phi_i)(\nabla\Phi_{i+1}-\nabla\Phi_i)||\le C(n,N,m)\delta.\\
||L(\nabla \Phi_i)^{-1}L(\nabla \Phi_{i+1}) - I_k||\le C(n,N,m)\delta.
\end{gathered}  
\end{equation}

\enu

\end{lem}

A direct consequence of Lemma \ref{C^1 and C^2 modulo distortion} is a bound of the transformation matrix $L(\nabla \Phi_i)$ in terms of the scale $r_i$. This also implies the uniform convergence of $\Phi_i$.

\begin{cor}\label{bound of L(Phi_i)}

Assume the same condition as in Lemma \ref{C^1 and C^2 modulo distortion}. For $\alpha>0$, if $\delta<\delta(n,N,m,\alpha)$, then the following is true.

\bnu 

\item $|L(\nabla \Phi_i)|_{op}, |L(\nabla \Phi_i)^{-1}|_{op}\le r_i^{-\alpha}$.

\item For $j>i$, $||\Phi_{j}-\Phi_{i}||_{L^\infty}\le C(n,N,m)\delta r_i^{1-\alpha}.$
\enu

\end{cor}

Thus we may define the global map $\Phi:\Rb^n\to \Rb^k$ as the limit of $\Phi_i$.

\begin{dfn}\label{definition of Phi}

Assume the same conditions and constructions as in Corollary \ref{bound of L(Phi_i)}. Then the maps $\Phi_i$ converges in $L^\infty$. We define the limit as $\Phi:=\lim_i\Phi_i$.
    
\end{dfn}

By Corollary \ref{bound of L(Phi_i)}, we get the following control of $\Phi$ immediately.

\begin{prop}\label{bound of Phi}
Assume the same condition as in Corollary \ref{bound of L(Phi_i)}. Then
\bnu
\item For each $i$, $||\Phi_i-\Phi||_{L^\infty}\le C(n,N,m)\delta r_i^{1-\alpha}$.
\item For any $x\in \Rb^n$, $|\Phi(x)-x|\le C(n,N,m)\delta$.

\enu
\end{prop}

We will prove Lemma \ref{C^1 and C^2 modulo distortion} by induction. The core observation is the following claim:

\begin{clm}\label{induction lemma for C^2 control}

Pick $x_{i+1,b}\in B_{0.01r_i}(x_{i,a})$. Assume there is a sufficiently small $0<\epsilon_{i,a}<\epsilon(n)$ such that
\bnu
\item $\Phi_i$ is $\epsilon_{i,a} r_i^{-1}$-regular in $B_{0.4r_i}(x_{i,a})$.
\item For $x\in B_{0.4r_i}(x_{i,a})$, $d((\ker \nabla \Phi_i(x))^\perp,\hat{l}_{x,r_i})\le \epsilon_{i,a}$.
\item $\diam{A^\perp_{x_{i+1,b},r_{i+1}}}\le \beta r_{i+1}$.

\enu

Define $\Psi_{i+1}$ as in (\ref{definition of Psi_i}). Then there is some $C(n)$ and $\epsilon_{i+1,b}=C(n)((2^{-m}+\beta)\epsilon_{i,a}+m\delta)$ such that the following holds in $B_{0.4r_{i+1}}(x_{i+1,b})$:

\bnu

\item $\Psi_{i+1}$ is $\epsilon_{i+1,b} r_{i+1}^{-1}$-regular. 
\item $d((\ker \nabla \Psi_{i+1}(x))^\perp,\hat{l}_{x,r_{i+1}})\le \epsilon_{i+1,b}$.
\item $|L(\nabla\Phi_i(x_{i+1,b}))^{-1}L(\nabla \Psi_{i+1}(x))-I_k|\le \epsilon_{i+1,b}$.

\enu

\end{clm}

\pf

\textbf{Step 1:} In $B_{0.4r_{i+1}}(x_{i+1,b})$, we have
\begin{equation}
\Psi_{i+1}=\sum_{c}\psi_{i+1,c}f_{i+1,c},
\end{equation}
where the sums is only taken over those $c$ with $x_{i+1,c}\in B_{0.8r_{i+1}}(x_{i+1,b})$. 

Since $\sum_c \psi_{i+1,c}\equiv 1$, we have $\sum_c \nabla\psi_{i+1,c}\equiv 0$, $\sum_c \nabla^2\psi_{i+1,c}\equiv 0$. Since $f_{i+1,c}$ is affine, we have $\nabla^2 f_{i+1,c}=0$. Now compute
\begin{equation}\label{ker nabla Psi_i+1}
    \nabla\Psi_{i+1}(x)-\hat{f}_{i+1,b}=\sum_c \nabla \psi_{i+1,c} (f_{i+1,c}-f_{i+1,b})+\psi_{i+1,c}(\hat f_{i+1,c}-\hat{f}_{i+1,b}).
\end{equation}
For $j,l\le n$,
\begin{align}\label{nabla^2 ker Psi_i+1}
\begin{split}
\partial_j\partial_l\Psi_{i+1}&=\sum_c\partial
_j\partial_l\psi_{i+1,c}f_{i+1,c}+\partial_j\psi_{i+1,c}\partial_l f_{i+1,c}+\partial_l\psi_{i+1,c}\partial_j f_{i+1,c}+\psi_{i+1,c}\partial_j\partial_l f_{i+1,c}\\
&=\sum_c \partial_j\partial_l\psi_{i+1,c}(f_{i+1,c}-f_{i+1,b})+\partial_j\psi_{i+1,c}(\partial_l f_{i+1,c}-\partial_lf_{i+1,b})+\partial_l\psi_{i+1,c}(\partial_j f_{i+1,c}-\partial_j f_{i+1,b}).
\end{split}
\end{align}

Thus to control $L(\nabla\Psi_{i+1})\nabla^2\Psi_{i+1}$ and $\ker \nabla \Psi_{i+1}$, we need to control the terms in (\ref{ker nabla Psi_i+1}) and (\ref{nabla^2 ker Psi_i+1}). We will see that the $L(\nabla \Psi_{i+1})$ can be replaced by $L(\nabla \Phi_{i}(x_{i+1,b}))$ later, so it suffices to bound $L(\nabla \Phi_i(x_{i+1,b}))^{-1}(\hat{f}_{i+1,c}-\hat{f}_{i+1,b})$ and $L(\nabla \Phi_i(x_{i+1,b}))^{-1}(f_{i+1,c}-f_{i+1,b})$.

\vspace{10pt}
\textbf{Step 2:}  In this step we prove controls of the splitting directions near $x_{i+1,b}$ for later use. 

Let $A_{x_{i+1,b},r_{i+1}}$ be the $(k,\delta)$ splitting set of $S$ in $B_{r_{i+1}}(x_{i+1,b})$ with splitting direction $\hat{l}_{x_{i+1,b},r_{i+1}}$. Since $\diam{A^\perp_{x_{i+1,b},r_{i+1}}}\le \beta r$, all $x_{i+1,c}\in B_{0.8r_{i+1}}(x_{i+1,b})$ lies in $B_{(\beta + \delta)r_{i+1}}(\hat{l}_{x_{i+1,b},r_{i+1}})$. 

Recall that $r_{i+1}=2^{-m}r_i$. By Corollary \ref{comparison of splitting directions}, $d(\hat{l}_{x_{i+1,b},r_i},\hat{l}_{x_{i+1,b},r_{i+1}})\le 10 m\delta.$ Thus
\begin{equation}
d((\ker\nabla\Phi_i(x_{i+1,b}))^\perp,\hat{l}_{x_{i+1,b},r_{i+1}})\le \epsilon_{i,a} + 10m\delta.
\end{equation}
Or equivalently
\begin{equation}
d(\ker \nabla \Phi_i(x_{i+1,b}),\hat{l}^\perp_{x_{i+1,b},r_{i+1}})\le \epsilon_{i,a}+10m\delta.
\end{equation}

Note that $x\in B_{0.4r_{i+1}}(x_{i+1,b})$, $S\cap B_{r_{i+1}/2}(x)\neq \emptyset$. Hence Corollary \ref{comparison of splitting directions} implies
$d(\hat{l}_{x,r_{i+1}},\hat{l}_{x_{i+1,b},r_{i+1}})\le 40\delta$ for any $x\in B_{r_{i+1}}(x_{i+1,b})$. Recall that $\hat{P}_{i+1,c}$ is the orthogonal projection from $\Rb^n$ to $\hat{l}_{x_{i+1,c},r_{i+1}}$. By Lemma \ref{equivalent definition of distance of subspaces}, we have 
\begin{equation}
|\hat{P}_{i+1,c}-\hat{P}_{i+1,b}|_{op}\le 40 \delta.
\end{equation}

Additionaly, since $\Phi_i$ is $\epsilon_{i,a}$-regular, Lemma \ref{regular map equivalent definition} implies that 
\begin{equation}
||L(\nabla \Phi_i(x_{i+1,b}))^{-1}\nabla^2\Phi_i||_{L^\infty(B_{r_{i+1}}(x_{i+1,b}))}\le 2\epsilon_{i,a} r_i^{-1}.
\end{equation}

\textbf{Step 3:} In this step we give a bound of $L(\nabla \Phi_i(x_{i+1,b}))^{-1}(\hat{f}_{i+1,c}-\hat{f}_{i+1,b})$. 

Recall that $\hat{f}_{i+1,c}=\nabla \Phi_i(x_{i+1,c})\circ \hat{P}_{i+1,c}$. Compute
\begin{align}\begin{split}
&\,|L(\nabla \Phi_{i}(x_{i+1,b}))^{-1}(\hat{f}_{i+1,c}-\hat{f}_{i+1,b})|\\
=&\,|L(\nabla \Phi_{i}(x_{i+1,b}))^{-1}(\nabla \Phi_i(x_{i+1,c})\circ \hat{P}_{i+1,c}-\nabla\Phi_i(x_{i+1,b})\circ \hat{P}_{i+1,b})|\\
\le &\, |L(\nabla\Phi_{i}(x_{i+1,b}))^{-1}(\nabla \Phi_i(x_{i+1,c})-\nabla \Phi_i(x_{i+1,b}))\circ \hat{P}_{i+1,c}|\\
&\,+|L(\nabla \Phi_i(x_{i+1,b}))^{-1}\nabla \Phi_i(x_{i+1,b})\circ(\hat{P}_{i+1,c}-\hat{P}_{i+1,b})|\\
\le&\,2\epsilon_{i,a} (r_{i+1}/r_i)+40\delta\\
\le &\,C\cdot(2^{-m}\epsilon_{i,a}+\delta)
\end{split}\end{align}

\textbf{Step 4:} In this step we give a bound of $L(\nabla \Phi_i(x_{i+1,b}))^{-1}(f_{i+1,c}-f_{i+1,b})$. 

Recall that
\begin{equation}
f_{i+1,c}(x)=\hat{f}_{i+1,c}(x-x_{i+1,c})+\Phi_i(x_{i+1,c}).
\end{equation}

For $x\in B_{0.4r_{i+1}}(x_{i+1,b})$, we have
\begin{align}\label{f_i+1,c-f_i+1,b}
\begin{split}
f_{i+1,c}(x)-f_{i+1,b}(x)=&\,\Phi_i(x_{i+1,c})-\Phi_i(x_{i+1,b})+\hat{f}_{i+1,c}(x-x_{i+1,c})-\hat{f}_{i+1,b}(x-x_{i+1,b})\\
=&\,(\hat{f}_{i+1,c}-\hat{f}_{i+1,b})(x-x_{i+1,c})\\
&\,+(\Phi_{i}(x_{i+1,c})-\Phi_i(x_{i+1,b})-\nabla \Phi_i(x_{i+1,b})(x_{i+1,c}-x_{i+1,b})\\
&\,+(\nabla \Phi(x_{i+1,b}) - \hat{f}_{i+1,b})(x_{i+1,c}-x_{i+1,b}).
\end{split}
\end{align}
To bound $L(\nabla \Phi_{i}(x_{i+1,b}))^{-1}(f_{i+1,c}(x)-f_{i+1,b}(x))$, it suffices to bound each term in (\ref{f_i+1,c-f_i+1,b}). The estimates of the first two terms are straightforward.
\begin{align}
\begin{split}
&\,|L(\nabla\Phi_{i}(x_{i+1,b}))^{-1}(\hat{f}_{i+1,c}-\hat{f}_{i+1,b})(x-x_{i+1,c})|\\
\le &\,|L(\nabla\Phi_{i}(x_{i+1,b}))^{-1}(\hat{f}_{i+1,c}-\hat{f}_{i+1,b})|\cdot|x-x_{i+1,c}|\\
\le&\,(2\cdot 2^{-m}\epsilon_{i,a} + 40 \delta)\cdot 1.2r_{i+1}\\
=&\, (2.4\cdot 2^{-m}\epsilon_{i,a} +48\delta) r_{i+1}.
\end{split}
\end{align}
\begin{align}
\begin{split}
&\,|L(\nabla\Phi_{i}(x_{i+1,b}))^{-1}(\Phi_{i}(x_{i+1,c})-\Phi_i(x_{i+1,b})-\nabla \Phi_i(x_{i+1,b})(x_{i+1,c}-x_{i+1,b}))|\\
\le &\, \frac{1}{2}||L(\nabla \Phi_i(x_{i+1,b}))^{-1}\nabla^2\Phi_i||_{L^\infty(B_{r_{i+1}}(x_{i+1,b}))}\cdot|x_{i+1,c}-x_{i+1,b}|^2\\
\le &\, \epsilon_{i,a} r_{i+1}^2/r_i.
\end{split}
\end{align}
For the last term, recall that $x_{i+1,c}\in B_{(\beta + \delta)r_{i+1}}(l_{x_{i+1,b},r_{i+1}})\cap B_{r_{i+1}}(x_{i+1,b})$, where $l_{x_{i+1,b},r_{i+1}}$ is the $k$-plane in the direction of $\hat{l}_{x_{i+1,b},r_{i+1}}$ passing $x_{i+1,b}$. We have 
\begin{equation}|x_{i+1,c}-\hat{P}_{i+1,b}(x_{i+1,c})|\le (\beta + \delta)r_{i+1}.\end{equation} 
Note that $x_{i+1,c}-\hat{P}_{i+1,b}(x_{i+1,c})\in \hat{l}_{x_{i+1,b},r_{i+1}}^\perp$. Since we know $d(\ker \nabla \Phi_i(x),\hat{l}^\perp_{x_{i+1,b},r_{i+1}})\le \epsilon_{i,a}+10m\delta$, we have
\begin{align}
\begin{split}
&\,|(\nabla \Phi(x_{i+1,b}) - \hat{f}_{i+1,b})(x_{i+1,c}-x_{i+1,b})|\\
=&\,|L(\nabla \Phi_i(x_{i+1,b}))^{-1}\nabla \Phi(x_{i+1,b})\circ(\id - \hat{P}_{i+1,b})(x_{i+1,c}-x_{i+1,b})|\\
\le&\,|\pi(\nabla\Phi_i(x_{i+1,b}))(x_{i+1,c}-\hat{P}_{i+1,b}(x_{i+1,c}))|\\
\le&\,(\epsilon_{i,a}+10m\delta)|x_{i+1,c}-\hat{P}_{i+1,b}(x_{i+1,c})|\\
\le&\,(\epsilon_{i,a}+10m\delta)(\beta+\delta)r_{i+1}.
\end{split}\end{align}

Sum everything up, we have
\begin{align}\begin{split}
|L(\nabla \Phi_{i}(x_{i+1,b}))^{-1}(f_{i+1,c}(x)-f_{i+1,b}(x))|&\le ((3.4\cdot 2^{-m}+\beta+\delta)\epsilon_{i,a}+(10m(\beta + \delta)+48)\delta)r_{i+1}\\
&\le C\cdot((2^{-m}+\beta)\epsilon_{i,a}+m\delta)r_{i+1}.
\end{split}\end{align}

\textbf{Step 5:} In this step we analyze the subspace $\ker\nabla \Psi_{i+1}(x)$.

For $x\in B_{0.4r_{i+1}}(x_{i+1,b})$, compute
\begin{align}
\begin{split}
&\,|L(\nabla \Phi_{i}(x_{i+1,b}))^{-1}(\nabla\Psi_{i+1}(x)-\hat{f}_{i+1,b})|\\
=&\,|\sum_c \nabla \psi_{i+1,c} L(\nabla \Phi_{i}(x_{i+1,b}))^{-1}(f_{i+1,c}-f_{i+1,b})+\psi_{i+1,c}L(\nabla \Phi_{i}(x_{i,b}))^{-1}(\hat f_{i+1,c}-\hat{f}_{i+1,b})|\\
\le&\, C(n)r_{i+1}^{-1}((2^{-m}+\beta+\delta)\epsilon_{i,a}+m\delta)r_{i+1}+C(n)(2^{-m}\epsilon_{i,a}+\delta)\\
\le&\,C(n)((2^{-m}+\beta)\epsilon_{i,a}+m\delta).
\end{split}
\end{align}

Recall that $\hat{f}_{i+1,b} = \nabla\Phi_i(x_{i+1,b})\circ \hat{P}_{i+1,b}$ and
\begin{equation}
d(\ker \nabla \Phi_i(x_{i+1,b}),\hat{l}^\perp_{x_{i+1,b},r_{i+1}})\le \epsilon_{i,a}+10m\delta.
\end{equation}
We know
\begin{equation}\label{nabla Phi_i - f_i+1}
    |L(\nabla \Phi_i(x_{i+1,b}))^{-1}(\nabla\Phi_i(x_{i+1,b})-\hat{f}_{i+1,b})|\le \epsilon_{i,a}+10m\delta.
\end{equation}
By Lemma \ref{smoothness of QR cor1}, we have
\begin{equation}\label{Phi_i and f_i+1}
    |L(\nabla \Phi_i(x_{i+1,b}))^{-1}L(\hat{f}_{i+1,b})-I_k|\le C(n)(\epsilon_{i,a}+10m\delta)<0.01.
\end{equation}
Thus the $L(\nabla \Phi_i(x_{i+1,b}))^{-1}$ term can be replaced by $L(\hat{f}_{i+1,b})^{-1}$.
\begin{equation}
    |L(\hat{f}_{i+1,b})^{-1}(\nabla\Psi_{i+1}(x)-\hat{f}_{i+1,b})|\le C(n)((2^{-m}+\beta)\epsilon_{i,a}+m\delta).
\end{equation}

Apply Lemma \ref{smoothness of QR cor1}, we yield 
\begin{equation}\label{Psi_i+1 and f_i+1}
\begin{split}
|\pi(\nabla \Psi_{i+1}(x))-\pi({\hat{f}_{i+1,b}})|& \le C(n)((2^{-m}+\beta)\epsilon_{i,a}+m\delta),\\
|L(\hat{f}_{i+1,b})^{-1}L(\Psi_{i+1}(x))-&I_k| \le C(n)((2^{-m}+\beta)\epsilon_{i,a}+m\delta).
\end{split}
\end{equation}
Note that $\hat{l}_{x_{i+1,b},r_{i+1}}=(\ker \hat{f}_{i+1,b})^{\perp}$. We get
\begin{equation}d((\ker\nabla \Psi_{i+1}(x))^\perp,\hat{l}_{x_{i+1,b},r_{i+1}})\le C(n)((2^{-m}+\beta)\epsilon_{i,a}+m\delta),\end{equation}
which is the main estimate of this step.

Combine (\ref{Phi_i and f_i+1}) and (\ref{Psi_i+1 and f_i+1}), we also get
\begin{equation}\label{Phi_i and Psi_i+1}
    |L(\nabla\Phi_i(x_{i+1,b}))^{-1}L(\nabla \Psi_{i+1}(x))-I_k|\le C(n)((2^{-m}+\beta)\epsilon_{i,a}+m\delta).
\end{equation}

\textbf{Step 6:} In this step we prove the regularity of $\Psi_{i+1}$.

In $B_{0.4r_{i+1}}(x_{i+1,b})$, for any $j,l\le n$, from (\ref{nabla^2 ker Psi_i+1}) we have
% compute
% \begin{equation}
% \partial_j\Psi_{i+1}=\sum_c\partial_j \psi_{i+1,c}f_{i+1,c}+\psi_{i+1,c}\partial_j f_{i+1,c},
% \end{equation}
% \begin{equation}
% \partial_j\partial_l\Psi_{i+1}=\sum_c\partial
% _j\partial_l\psi_{i+1,c}f_{i+1,c}+\partial_j\psi_{i+1,c}\partial_l f_{i+1,c}+\partial_l\psi_{i+1,c}\partial_j f_{i+1,c}+\psi_{i+1,c}\partial_j\partial_l f_{i+1,c}.
% \end{equation}
% Note that $\sum_c \psi_{i+1,c}\equiv 1$, we have $\sum_c \partial_j\psi_{i+1,c}\equiv 0$ and $\sum_c \partial_j\partial_l\psi_{i+1,c}\equiv 0$. Since $f_{i+1,c}$ is affine, $\partial_j\partial_l f_{i+1,c}=0$. Thus
\begin{align}
\begin{split}
&\,r_{i+1}|L(\nabla \Phi_{i}(x_{i+1,b}))^{-1}\partial_j\partial_l\Psi_{i+1}|\\
=&\,r_{i+1}|\sum_c L(\nabla \Phi_{i}(x_{i+1,b}))^{-1}(\partial_j\partial_l\psi_{i+1,c}(f_{i+1,c}-f_{i+1,b})\\
&\,+\partial_j\psi_{i+1,c}(\partial_l f_{i+1,c}-\partial_lf_{i+1,b})+\partial_l\psi_{i+1,c}(\partial_j f_{i+1,c}-\partial_j f_{i+1,b}))|\\
\le&\,r_{i+1}C(n)r_{i+1}^{-2}\cdot C\cdot((2^{-m}+\beta+\delta)\epsilon_{i,a}+m\delta)r_{i+1}+r_{i+1}C(n)r_{i+1}^{-1}\cdot C\cdot(2^{-m}\epsilon_{i,a}+\delta)\\
\le&\,C(n)((2^{-m}+\beta)\epsilon_{i,a}+m\delta)
\end{split}
\end{align}
Therefore, 
\begin{equation}||L(\nabla \Phi_{i}(x_{i+1,b}))^{-1}\nabla^2\Psi_{i+1}||\le C(n)((2^{-m}+\beta)\epsilon_{i,a}+m\delta)r_{i+1}^{-1}\end{equation}
By (\ref{Phi_i and Psi_i+1}), we may replace $L(\nabla \Phi_i(x_{i+1,b}))^{-1}$ by $L(\nabla \Psi_{i+1}(x))^{-1}$ and have
\begin{equation}\label{Psi_i+1 regular}
    ||L(\nabla \Psi_{i+1})^{-1}\nabla^2\Psi_{i+1}||\le C(n)((2^{-m}+\beta)\epsilon_{i,a}+m\delta)r_{i+1}^{-1}
\end{equation}
Take $\epsilon_{i+1,b}=C(n)((2^{-m}+\beta)\epsilon_{i,a}+m\delta)$ for some appropriate $C(n)$ finishes the proof of the claim.

\epf

Now we prove Lemma \ref{C^1 and C^2 modulo distortion} and Corollary \ref{bound of L(Phi_i)}. The proofs of these results are simply gathering up the partial results in Claim \ref{induction lemma for C^2 control}.

\pf[Proof of Lemma \ref{C^1 and C^2 modulo distortion}] 

It is clear that $\Phi_1$ satisfies all the estimates. Take a sequence of balls $B_{r_i}(x_{i,a_i})$ with $x_{i+1,a_{i+1}}\in B_{0.01 r_i}(x_{i,a_i})$ for each $i$. We know $\Phi_i=\Psi_i$ in $B_{r_{i+1}}(x_{i+1,a_{i+1}})$. Consider the constants $\epsilon_{i,a_{i}}$. By Claim \ref{induction lemma for C^2 control}, if $\diam{A^\perp_{x_{i+1,a_{i+1}},r_{i+1}}}< 2^{-m+2}$, we have
\begin{equation}\epsilon_{i+1,a_{i+1}}\le C(n)(2^{-m}\epsilon_{i,a_i}+m\delta).\end{equation}
Otherwise, we have
\begin{equation}\epsilon_{i+1,a_{i+1}}\le C(n)(\epsilon_{i,a}+m\delta).\end{equation}

By Proposition \ref{only finite scales have large width}, given $M>0$, if $\delta<\delta(n,N,m,M)$, there are at most $N$ different $i$ with $\diam{A^\perp_{x_{i+1,a_{i+1}},r_{i+1}}} \ge 2^{-m+2}$ in each $M$ subsequent scales. Assume $m>m(n)$, $M>M(n,N,m)$, $\delta<\delta(n,N,m,M)$. Then it is easy to verify that the sequence $\{\epsilon_{i,a_i}\}$ is bounded with $\epsilon_{i,a_i}\le C(n,N,m,M)\delta$. Fix $M$ and we have
\begin{equation}
    \epsilon_{i,a_i}\le C(n,N,m)\delta.
\end{equation}

For $x\in B_{0.4 r_{i+1}}(x_{i+1,b})$, the above discussion has proved that
\begin{equation}
\begin{split}
&\quad||L(\nabla \Phi_i)^{-1}\nabla^2 \Psi_{i+1}||\le C(n,N,m)\delta r_{i+1}^{-1},\\
&\,\,||L(\nabla \Phi_i)^{-1}(\nabla\Psi_{i+1}-\hat{f}_{i+1,b})||\le C(n,N,m)\delta,\\
&||L(\nabla \Phi_i)^{-1}(\Psi_{i+1}-f_{i+1,b})||\le C(n,N,m)\delta r_{i+1}.
\end{split}
\end{equation}

To finish the proof, it suffices to replace $\Psi_{i+1}$ by $\Phi_{i+1}$. Recall that
\begin{equation}
    \Phi_{i+1} = \chi_{i+1}\Psi_{i+1}+(1-\chi_{i+1})\Phi_i.
\end{equation}

In $B_{0.4r_{i+1}}(x_{i+1,b})$, $\Phi_i=\Psi_i$. By (\ref{Psi_i+1 regular}), we have
\begin{equation}
    ||L(\nabla \Phi_i)^{-1}\nabla^2\Phi_i||\le C(n,N,m)\delta r_i^{-1}.
\end{equation}
and thus
\begin{equation}
    ||L(\nabla \Phi_i)^{-1}L(\nabla\Phi_i(x_{i+1,b}))-I_k||\le C(n,N,m)\delta.
\end{equation}

Note that (\ref{nabla Phi_i - f_i+1}) now reads
\begin{equation}
    |L(\nabla \Phi_i(x_{i+1,b}))^{-1}(\nabla \Phi_i(x_{i+1,b})-\hat{f}_{i+1,b})|\le C(n,N,m)\delta.
\end{equation}

Thus we can replace $L(\nabla\Phi_i(x_{i+1,b}))$ by $L(\nabla \Phi_i)$ and get
\begin{align}\label{Phi_i and f_i+1,b}
\begin{split}
    ||&L(\nabla \Phi_i)^{-1}(\nabla\Phi_i-\hat{f}_{i+1,b})||\le C(n,N,m)\delta r_{i+1}.\\
    &||L(\nabla \Phi_i)^{-1}(\Phi_i-f_{i+1,b})||\le C(n,N,m)\delta r_{i+1}.
\end{split}
\end{align}

Take the derivative of $\Phi_{i+1}$ and combine the estimates above. Thus we have
\begin{equation}\label{first estimate of Phi_i+1}
\begin{split}
&\quad||L(\nabla \Phi_i)^{-1}\nabla^2 \Phi_{i+1}||\le C(n,N,m)\delta r_{i+1}^{-1},\\
&\,\,||L(\nabla \Phi_i)^{-1}(\nabla\Phi_{i+1}-\hat{f}_{i+1,b})||\le C(n,N,m)\delta,\\
&||L(\nabla \Phi_i)^{-1}(\Phi_{i+1}-f_{i+1,b})||\le C(n,N,m)\delta r_{i+1}.
\end{split}
\end{equation}
Combine (\ref{Phi_i and f_i+1,b}) and (\ref{first estimate of Phi_i+1}). We have
\begin{equation}\label{Phi_i and Phi_i+1}
\begin{gathered}
||L(\nabla \Phi_i)(\Phi_{i+1}-\Phi_i)||\le C(n,N,m)\delta r_{i+1},\\
||L(\nabla \Phi_i)(\nabla\Phi_{i+1}-\nabla\Phi_i)||\le C(n,N,m)\delta.
\end{gathered}  
\end{equation}
And by (\ref{Phi_i and Phi_i+1}) we have
\begin{equation}\label{transformation of Phi_i and Phi_i+1}
||L(\nabla\Phi_i)L(\nabla \Phi_{i+1})-I_k||\le C(n,N,m)\delta.
\end{equation}

By (\ref{Phi_i and Phi_i+1}), we may replace $L(\nabla \Phi_{i})^{-1}$ in (\ref{first estimate of Phi_i+1}) by $L(\nabla \Phi_{i+1})^{-1}$. Thus in $B_{0.4r_{i+1}}(x_{i+1,b})$ we have
\begin{equation}
    ||L(\nabla \Phi_{i+1})^{-1}\nabla^2\Phi_{i+1}||\le C(n,N,m)\delta r_{i+1}^{-1}.
\end{equation}

We have been proving the estimates in $B_{0.4r_{i+1}}(x_{i+1,b})$. However, since $\Phi_{i+1}=\Phi_i$ outside $\bigcup_aB_{0.4r_{i+1}}(x_{i+1,a})$, (\ref{Phi_i and Phi_i+1}) and (\ref{transformation of Phi_i and Phi_i+1}) actually hold on the entire $\Rb^n$. And we may apply the estimates for $||L(\nabla\Phi_i)^{-1}\nabla^2\Phi_i||$ to $||L(\nabla\Phi_{i+1})^{-1}\nabla^2\Phi_{i+1}||$. Thus inductively we see that $\Phi_i$ is $C(n,N,m)\delta r_i^{-1}$-regular in $\Rb^n$. Finally, by (\ref{first estimate of Phi_i+1}) and (\ref{transformation of Phi_i and Phi_i+1}), in $\bigcup_aB_{0.4r_{i+1}}(x_{i+1,a})$ we have
\begin{equation}
d(\ker\nabla\Phi_{i+1}(x),\hat{l}^\perp_{x,r_{i+1}})\le C(n,N,m)\delta.
\end{equation}

We've finished the proof of Lemma \ref{C^1 and C^2 modulo distortion}.

\epf

\pf[Proof of Corollary \ref{bound of L(Phi_i)}]

\bnu

\item By 3. of Lemma \ref{C^1 and C^2 modulo distortion}, we have
\begin{equation}
    |L(\nabla \Phi_i)^{-1}L(\nabla \Phi_{i+1})-I_k|\le C(n,N,m)\delta.
\end{equation}
Note that $L(\nabla \Phi_0)=I_k$. We have
\begin{equation}
    |L(\nabla \Phi_i)|_{op}, |L(\nabla \Phi_i)^{-1}|_{op}\le (1+C(n,N,m)\delta)^i.
\end{equation}
Taking $\delta$ sufficiently small, we have
\begin{equation}
    1+C(n,N,m)\delta\le 2^{m\alpha}.
\end{equation}
The conclusion follows immediately.

\item By 3. of Lemma \ref{C^1 and C^2 modulo distortion} and 1.,
\begin{equation}
    ||\Phi_{i+1}(x)-\Phi_i(x)||\le C(n,N,m)\delta r_i^{1-\alpha}.
\end{equation}

For general $j>i$, we have
\begin{equation}
    ||\Phi_j-\Phi_i||\le \sum_{k=i}^{\infty}||\Phi_{k+1}-\Phi_k||\le \sum_{k=i}^{\infty}C(n,N,m)\delta r_k^{1-\alpha}\le C(n,N,m)\delta r_i^{1-\alpha}.
\end{equation}
\enu

\epf

% Since $B_{0.3r_i}(S\cap B_{1.99}(0^n))\subset \bigcup_a B_{0.4 r_i}(x_{i,a})$. We have actually proved Lemma \ref{C^1 and C^2 modulo distortion} for $x\in B_{0.3r_i}(S\cap B_{1.99}(0^n))$. It then remains to prove Lemma \ref{C^1 and C^2 modulo distortion} itself.

% We first prove the regularity of $\Phi_i$ inductively. We simply need to repeat the estimates in claim \ref{induction lemma for C^2 control}. The base case $i=0$ is trivial. Assume the conclusion holds for $i$ and consider the case for $i+1$. For $x\in \Rb^n$, if $d(x,(x_{i+1,a})>0.5r_{i+1}$ for any $a$, then $\psi_{i+1}(x) = 1$ near $x$. The conclusion follows from the inductive hypothesis since $\Phi_{i+1}=\Phi_i$ near $x$. Now assume $d(x,x_{i+1,a})\le 0.5 r_{i+1}$. In $B_{r_{i+1}}(x_{i+1,a})$, we have 

\subsection{Level Sets of $\Phi$}

In this subsection, we'll prove the H\"older estimates for $\Phi$ and level sets of $\Phi$. We list the main results of this subsection at the start, before diving into the proofs of them which are somehow intertwined.

The most important result in this subsection is the following lemma. It says that the map $c\to \Phi^{-1}(c)$ is $C^{1-\alpha}$-biH\"older, with a control of $dist$ for the lower bound.
\begin{prop}\label{distance between level sets of Phi}
Assume the same condition as in Corollary \ref{bound of L(Phi_i)}.  Then for any $c,d\in \Rb^n$ with $|c-d|\le 1$, we have
\begin{equation}
\begin{split}
    dist(\Phi^{-1}(c),\Phi^{-1}(d))&\ge (1-C(n,N,m)\delta)|c-d|^{\frac{1}{1-\alpha}},\\
    d_H(\Phi^{-1}(c),\Phi^{-1}(d))&\le (1+C(n,N,m)\delta)|c-d|^{\frac{1}{1+\alpha}}.
\end{split}
\end{equation}

\end{prop}

The lower bound of $dist$ can be rephrase as the $C^{1-\alpha}$ H\"older continuity of $\Phi$.

\begin{cor}\label{Phi is Holder}

Assume the same conditions as in Corollary \ref{bound of L(Phi_i)}. Then for $x,y\in \Rb^n$ with $|x-y|\le 1$, we have
\begin{equation}
    |\Phi(x)-\Phi(y)|\le (1+C(n,N,m)\delta)|x-y|^{1-\alpha}.
\end{equation}
    
\end{cor}

We'll also see in this subsection that each level set of $\Phi$ is a $C^{1-\alpha}$ biH\"older manifold.

\begin{prop}\label{level set is biholder}

Assume the same conditions as in Corollary \ref{bound of L(Phi_i)}. Then for each $c\in \Rb^k$, $\Phi^{-1}(c)$ is a $C^{0,1-\alpha}$-H\"older manifold.
    
\end{prop}

As we will see later, all these results follow from their respective versions for $\Phi_i$. Thus we will first focus on $\Phi_i$ and $\Phi_i^{-1}(c)$.

Now let's prove the results above. A first consequence of the estimates from the last subsection is that each level set of $\Phi_i$ is a graphical submanifold.

\begin{lem}\label{level sets of phi_i is graphical}

Assume the same condition as in Corollary \ref{bound of L(Phi_i)}. Then for each $c\in \Rb^k$, $\Phi_i^{-1}(c)$ is a $k$-dimensional $(C(n,N,m)\delta, r_i)$-graphical manifold.

\end{lem}

\pf

By Lemma \ref{C^1 and C^2 modulo distortion}, $\Phi_i$ is $C(n,N,m)\delta r_i^{-1}$-regular. Thus the conclusion now follows from Proposition \ref{level sets of delta_x regular maps}.
\epf

The following lemma proves that level sets of $\Phi_{i}$ and $\Phi_{i-1}$ are $C(n,N,m)\delta$-close, and the closest projection map between them is a $(1+C(n,N,m)\delta)$-biLipschitz diffeomorphism.

\begin{lem}\label{projection between level sets of phi}
Assume the same conditions as in Corollary \ref{bound of L(Phi_i)}. For $c\in \Rb^k$, let $P_{i,c}:\Phi_{i-1}^{-1}(c)\to \Phi_i^{-1}(c)$ be the closest point projection map. Then each $P_{i,c}$ is a diffeomorphism satisfying the following.
\bnu
\item For any $x\in \Phi_{i-1}^{-1}(c)$, $|P_{i,c}(x)-x|\le C(n,N,m)\delta r_{i-1}$.
\item $P_{i,c}$ is $(1+C(n,N,m)\delta)$-biLipschitz.
\enu
\end{lem}
\pf
Recall that $\Phi_{i-1}$, $\Phi_i$ are both $C(n,N,m)\delta r^{-1}_{i-1}$-regular. By (\ref{nabla Phi_i - f_i+1}), (\ref{Phi_i and f_i+1,b}), (\ref{first estimate of Phi_i+1}), we have
\begin{equation}
\begin{split}
    |L(\nabla \Phi_{i-1})(\Phi_i-\Phi_{i-1})|\le C(n,N,m)\delta r_{i-1},\\
    |L(\nabla \Phi_{i-1})(\nabla\Phi_i-\nabla\Phi_{i-1})|\le C(n,N,m)\delta.
\end{split}
\end{equation}
The conclusion now follows from Proposition \ref{projection between level sets is diffeomorphism}.
\epf

Combining the conclusions of Lemma \ref{projection between level sets of phi}, we have the following corollary.
\begin{cor}
Assume the same conditions as in Lemma \ref{projection between level sets of phi}. Then for $x,y\in \Phi^{-1}_{i-1}(c)$,
\begin{equation}
||P_{i,c}(x)-P_{i,c}(y)|-|x-y||\le C(n,N,m)\delta\cdot \min\{|x-y|,r_{i-1}\}.
\end{equation}
\end{cor}

Another unsurprising consequence of Lemma \ref{projection between level sets of phi} is that $\Phi^{-1}_i(c)$ converges to $\Phi^{-1}(c)$ as $i\to \infty$. It is clear that the sequence is Cauchy, but not as straightforward to show the limit is $\Phi^{-1}(c)$.

\begin{lem}\label{level set of Phi_i and Phi}

Assume the same condition as in Corollary \ref{bound of L(Phi_i)}. 

\bnu

\item For any $c\in \Rb^k$, $d_H(\Phi_{i}^{-1}(c),\Phi_{i-1}^{-1}(c))\le C(n,N,m)\delta r_{i-1}$.

\item The level sets of $\Phi_i$ converges in the Hausdorff distance: $\Phi_i^{-1}(c)\xrightarrow{\mathrm{H}}\Phi^{-1}(c)$, with 
\begin{equation}
    d_H(\Phi_i^{-1}(c),\Phi^{-1}(c))\le C(n,N,m)\delta r_i.
\end{equation}

\enu
    
\end{lem}

To prove Lemma \ref{level set of Phi_i and Phi}, we need the following version of Proposition \ref{distance between level sets of Phi} for level sets of $\Phi_j$. One should note that the estimates don't depend on $j$.

\begin{lem}\label{distance between level sets of Phi_i}

Assume the same condition as in Corollary \ref{bound of L(Phi_i)}. Then for any $c,d\in \Rb^n$ with $|c-d|\le 1$, for any $j$, we have
\begin{equation}
\begin{split}
    dist(\Phi_j^{-1}(c),\Phi_j^{-1}(d))&\ge (1-C(n,N,m)\delta)|c-d|^{\frac{1}{1-\alpha}},\\
    d_H(\Phi_j^{-1}(c),\Phi_j^{-1}(d))&\le (1+C(n,N,m)\delta)|c-d|^{\frac{1}{1+\alpha}}.
\end{split}
\end{equation}

\end{lem}

\pf

First we prove the lower bound. For $x\in \Phi_j^{-1}(c)$, $y\in \Phi_j^{-1}(d)$, if $|x-y|\ge 1$, there is nothing to prove. If $|x-y|\le r_j$, by Corollary \ref{bound of L(Phi_i)}, $||L(\nabla\Phi_i)||\le r_i^{-\alpha}$. Thus
\begin{equation}
    |c-d|\le r_j^{-\alpha}|x-y|\le |x-y|^{1-\alpha}.
\end{equation}
In this case, the conclusion is immediate.

Assume $r_j<|x-y|\le 1$. Then there exists a unique $0\le i<j$ such that
\begin{equation}
    r_{i}\ge |x-y|> r_{i+1}.
\end{equation} 
By Lemma \ref{projection between level sets of phi}, there exists $x_i\in \Phi_i^{-1}(c)$, $y_i\in \Phi_i^{-1}(d)$, with $|x_i-x|\le C(n,N,m)\delta r_i$, $|y_i-y|\le C(n,N,m)\delta r_i$. Since $||L(\nabla\Phi_i)||\le r_i^{-\alpha}$,
\begin{equation}
    |c-d|\le r_i^{-\alpha}|x_i-y_i|\le (1+C(n,N,m)\delta)r_i^{-\alpha}|x-y|\le (1+C(n,N,m)\delta)|x-y|^{1-\alpha}.
\end{equation}
Or equivalently, $|x-y|\ge (1-C(n,N,m))|c-d|^{\frac{1}{1-\alpha}}$.

Now we consider the upper bound. For each $i \ge j$, we have $d_H(\Phi_i^{-1}(c),\Phi_j^{-1}(c))\le C(n,N,m)\delta r_i$, $d_H(\Phi_i^{-1}(d),\Phi_j^{-1}(d))\le C(n,N,m)\delta r_i$. Note that $||L(\nabla \Phi_i)^{-1}||\le r_i^{-\alpha}$. By Proposition \ref{distance between level sets of regular map}, we have $d_H(\Phi_i^{-1}(c),\Phi_i^{-1}(d))\le r_i^{-\alpha}|c-d|$. Combine these together and we get
\begin{equation}\label{potential upperbound of distance between level sets}
    d_H(\Phi_j^{-1}(c),\Phi_j^{-1}(d))\le r_i^{-\alpha}|c-d|+C(n,N,m)\delta r_i.
\end{equation}
Apply Proposition \ref{distance between level sets of regular map} directly for $j$, we have
\begin{equation}
    d_H(\Phi_j^{-1}(c),\Phi_j^{-1}(d))\le r_j^{-\alpha}|c-d|.
\end{equation}
If $|c-d|\le r_j^{1+\alpha}$, then
\begin{equation}
    d_H(\Phi_j^{-1}(c),\Phi_j^{-1}(d))\le r_j^{-\alpha}|c-d|\le |c-d|^{\frac{1}{1+\alpha}}.
\end{equation}
Otherwise, there exists $i>j$ such that $r_{i+1}^{1+\alpha}<|c-d|\le r_{i}^{1+\alpha}$. Hence
\begin{align}
\begin{split}
    d_H(\Phi_j^{-1}(c),\Phi_j^{-1}(d))&\le r_i^{-\alpha}|c-d|+C(n,N,m)\delta r_i\\
    &\le |c-d|^{-\frac{\alpha}{1+\alpha}}|c-d|+C(n,N,m)\delta r_{i+1}\\
    &\le |c-d|^{\frac{1}{1+\alpha}}+C(n,N,m)\delta |c-d|^{\frac{1}{1+\alpha}}\\
    &=(1+C(n,N,m)\delta)|c-d|^{\frac{1}{1+\alpha}}.   
\end{split}
\end{align}

\epf

Now we can prove Lemma \ref{level set of Phi_i and Phi}.

\pf[Proof of Lemma \ref{level set of Phi_i and Phi}]

1. is a direct consequence of Lemma \ref{projection between level sets of phi}. For 2., we only need to check the Hausdorff limit of $\Phi_i^{-1}$(c) is indeed $\Phi^{-1}(c)$.

Assume $x_i\in \Phi^{-1}_i(c)$ and $x_i\to x$. We'll show $x\in \Phi^{-1}(c)$. Since $\Phi_i\to \Phi$ uniformly, $\Phi$ is continuous and $\Phi(x_i)\to \Phi(x)$. By Proposition \ref{bound of Phi}, 
\begin{equation}
    |\Phi(x_i)-c|=|\Phi(x_i)-\Phi_i(x_i)|\le||\Phi-\Phi_i||\le C(n,N,m)\delta r_i^{1-\alpha}.
\end{equation}
Thus $\lim_{i\to \infty}\Phi(x_i)=c$ and $\Phi(x) = c$.

Conversely, let $x\in \Phi^{-1}(c)$ and we'll show $x$ is in the limit of $\Phi^{-1}_i(c)$. For any $i$, we have $|\Phi_i(x)-c|\le ||\Phi_i-\Phi||\le C(n,N,m)\delta r_i^{1-\alpha}$. By Lemma \ref{distance between level sets of Phi_i}, there exists $y\in \Phi_i^{-1}(c)$ with $|x-y|\le (1+C(n,N,m)\delta)(C(n,N,m)\delta r_i^{1-\alpha})^{\frac{1}{1+\alpha}}\le C(n,N,m)(\delta r_{i}^{1-\alpha})^{\frac{1}{1+\alpha}}$. Thus
\begin{equation}
    d(x, \Phi^{-1}_i(c))\le C(n,N,m)(\delta r_{i}^{1-\alpha})^{\frac{1}{1+\alpha}},
\end{equation}
i.e., $d(x,\Phi_i^{-1}(c))\to 0$ and $x$ is in the limit of $\Phi_i^{-1}(c)$. Therefore, $d_H(\Phi_i^{-1}(c),\Phi^{-1}(c))\to 0$.

\epf

We have proved Lemma \ref{level set of Phi_i and Phi} using the partial version of Proposition \ref{distance between level sets of Phi}. Interestingly, Lemma \ref{level set of Phi_i and Phi} is only remaining ingredient for the proof of Proposition \ref{distance between level sets of Phi}.

\pf[Proof of Proposition \ref{distance between level sets of Phi}]

By Lemma \ref{level set of Phi_i and Phi}, $\Phi_i^{-1}(c) \xrightarrow{\mathrm{H}} \Phi^{-1}(c)$. Note that the estimates in Lemma \ref{distance between level sets of Phi_i} do not depend on $i$. We can simply pass them to the limit and yield Proposition \ref{distance between level sets of Phi}.

\epf

A direct corollary of Proposition \ref{distance between level sets of Phi} is that $\Phi$ is $(1-\alpha)$-H\"older.
\begin{proof}[Proof of Corollary \ref{Phi is Holder}]
For $x,y\in \Rb^n$ with $|x-y|\le 1$, by Proposition \ref{bound of Phi}, 
\begin{equation}
    |\Phi(x)-x|,\,|\Phi(y)-y|\le C(n,N,m)\delta.
\end{equation}
For $|x-y|>1/2$, this implies
\begin{equation}
    |\Phi(x)-\Phi(y)|\le |x-y|+C(n,N,m)\delta\le (1+C(n,N,m)\delta)|x-y|^{1-\alpha}.
\end{equation}

For $|x-y|\le 1/2$, we have $|\Phi(x)-\Phi(y)|\le 1$. By Proposition \ref{distance between level sets of Phi}, we have
\begin{equation}
    |x-y|\ge (1-C(n,N,m)\delta)|\Phi(x)-\Phi(y)|^{\frac{1}{1-\alpha}},
\end{equation}
or equivalently,
\begin{equation}
    |\Phi(x)-\Phi(y)|\le (1+C(n,N,m)\delta)|x-y|^{1-\alpha}.
\end{equation}
\end{proof}

In the remaining part of this subsection, we will show that each $\Phi^{-1}(c)$ is a $C^{1-\alpha}$-biH\"older manifold. The idea is simple: the closest point projection from $\Phi_i^{-1}(c)$ to $\Phi_{i+1}^{-1}(c)$ is a diffeomorphism. By taking the composition, we get a map from $\Phi_0^{-1}(c)\to \Phi^{-1}(c)$. It remains to check this map is indeed a biH\"older homeomorphism.

\begin{lem}\label{level set of Phi is biHolder}
Assume the same conditions as in Lemma \ref{projection between level sets of phi}. For $c\in \Rb^k$, let $\Pi_{i,c}:=P_{i,c}\circ\cdots\circ P_{1,c}$. Let $\Pi_c:=\lim_i\Pi_{i,c}$. Then 

\bnu

\item $\Pi_c$ is well-defined and $\Pi_c(\Phi_0^{-1}(c)) = \Phi^{-1}(c)$.
\item $\Pi_c$ is biH\"older continuous. Assume $1-C(n,N,m)\delta\ge r_i^{\alpha}$. For $x,y\in \Rb^n$ with $|x-y|\le 1$,
\begin{equation}
    (1-C(n,N,m)\delta)|x-y|^{\frac{1}{1-\alpha}}\le|\Pi_c(x)-\Pi_c(y)|\le (1+C(n,N,m)\delta)|x-y|^{\frac{1}{1+\alpha}}.
\end{equation}

\enu

\end{lem}

\pf

For any $x\in \Phi_{i-1}^{-1}(c)$, $d(x, P_{i,c}(x))\le d(x,\Phi_{i}^{-1}(c))\le C(n,N,m)\delta r_i$. Thus the sequence $\Pi_{i,c}$ converges uniformly and $\Pi_c$ is well-defined. 

Next we prove $\Pi_c$ is bi-H\"older. Take $x,y\in \Rb^n$ with $|x-y|\le 1$. Let $d_0 = |x-y|$, $d_i=|\Pi_{i,c}(x)-\Pi_{i,c}(y)|$. Then
\begin{equation}\label{d_i to d_i+1}
    |d_{i} -d_{i-1}|\le C(n,N,m)\delta\cdot\min\{d_{i-1},r_{i-1}\}.
\end{equation}
For each $i$, we have
\begin{equation}
    (1-C(n,N,m)\delta)^i\le d_i/d_0\le (1+C(n,N,m)\delta)^i,
\end{equation}
\begin{equation}
    |d_i-d_\infty|\le C(n,N,m)\delta r_i.
\end{equation}
Recall that $1-C(n,N,m)\delta\ge r_1^{\alpha}$, $1+C(n,N,m)\delta\le r_1^{-\alpha}$. We have
\begin{equation}\label{bounds of d_infty}
    r_i^\alpha d_0-C(n,N,m)\delta r_i\le d_\infty\le r_i^{-\alpha}d_0+C(n,N,m)\delta r_i.
\end{equation}

It suffices to find the best $i$ for the above estimates. Note that the estimates above have the same expression as (\ref{potential upperbound of distance between level sets}), excepts that it holds for all $i$.

For the upper bound, pick $i$ such that
\begin{equation}
    r_{i+1}^{1+\alpha}<d_0\le r_i^{1+\alpha}.
\end{equation}
By the same argument, we have
\begin{equation}
d_\infty\le (1+C(n,N,m)\delta)d_0^{\frac{1}{1+\alpha}}.
\end{equation}

Similarly, for the lower bound, pick $i$ with
\begin{equation}
    r_{i+1}^{1-\alpha}<d_0\le r_i^{1-\alpha}.
\end{equation}
And we get
\begin{equation}
    d_\infty \ge (1-C(n,N,m)\delta)d_0^{\frac{1}{1-\alpha}}.
\end{equation}

We have proved the estimates, and it remains to check that $\Pi_c(\Phi_0^{-1}(c))=\Phi^{-1}(c)$. Note that $||\Pi_c-\Pi_{i,c}||\le C(n,N,m)\delta r_i$, $\Pi_{i,c}(\Phi^{-1}_0(c))=\Phi_{i}^{-1}(c)$, we have 
\begin{equation}
    d_H(\Pi_c(\Phi^{-1}_0(c)),\Phi^{-1}(c))\le ||\Pi_c-\Pi_{i,c}||+d(\Pi_{i,c}(\Phi_0^{-1}(c)), \Phi^{-1}(c))\le C(n,N,m)\delta r_i.
\end{equation}
Let $i\to \infty$, we get $d_H(\Pi_c(\Phi^{-1}_0(c)), \Phi^{-1}(c))=0$. 

Since $\Pi_c$ is biH\"older, $\Pi_c(\Phi_0^{-1}(c))$ is closed in $\Rb^n$. Since $\Phi^{-1}(c)$ and $\Pi_c(\Phi_0^{-1}(c))$ are both closed, we have $\Pi_c(\Phi_0^{-1}(c))=\Phi^{-1}(c)$.
\epf

Proposition \ref{level set is biholder} follows from Lemma \ref{level set of Phi is biHolder} immediately.

\subsection{Points of $S$ on Level Sets of $\Phi$}

In this subsection, we will prove the estimates for the intersection of $S$ and level sets of $\Phi$. These results will finish the proof of our main theorem.

\begin{prop}\label{distance between level sets of Phi intersecting S}

Assume the same condition as in Corollary \ref{bound of L(Phi_i)}. Then for any $c,d\in B_{1.98}(0^k)$, with $|c-d|\le 1$,
\begin{equation}
\begin{split}
    dist(\Phi^{-1}(c)\cap S, \Phi^{-1}(d)\cap S)\ge (1-C(n,N,m)\delta))|c-d|^{\frac{1}{1-\alpha}}.\\
    d_H(\Phi^{-1}(c)\cap S, \Phi^{-1}(d)\cap S)\le (1+C(n,N,m)\delta)|c-d|^{1-\alpha}.
\end{split}
\end{equation}
    
\end{prop}

The lower bound is a direct consequence of Proposition \ref{distance between level sets of Phi}. To prove the upper bound, we need the following technical lemmas:

The following lemma shows that the level sets of $\Phi$ can be approximated by a plane that is perpendicular to the splitting direction in a ball.
\begin{lem}\label{level sets of Phi are locally almost planes}
Assume the same conditions as in Corollary \ref{bound of L(Phi_i)}. For $i\ge 1$, $r_i=2^{-mi}$, assume $x\in \Phi^{-1}(c)$ with $B_{r_i}(x)\subset B_{1.99}(0^n)$, $S\cap B_{0.39r_i}(x)\neq\emptyset$. Let $\hat{l}_{x,r_i}$ be the splitting direction of $S$ in $B_{r_i}(x)$. Then
\begin{equation}
    d_H|_{B_{0.9r_i}(x)}(\Phi^{-1}(c),l_{x,r_i}^\perp)\le C(n,N,m)\delta r_i.
\end{equation}
\end{lem}
\pf

From the conditions we have, $x\in B_{0.39 r_i}(S\cap B_{1.99}(0^n))$. By Construction \ref{covering for all scales}, we see $x\in \bigcup_aB_{0.4 r_i}(x_{i,a})$. By Lemma \ref{C^1 and C^2 modulo distortion},
\begin{equation}
    d((\ker\nabla \Phi_i(x))^\perp,\hat{l}_{x,r_i})\le C(n,N,m)\delta.
\end{equation}
By Lemma \ref{level set of Phi_i and Phi}, $d_H(\Phi_i^{-1}(c),\Phi^{-1}(c))\le C(n,N,m)\delta r_i$ and there exists $x'\in \Phi_i^{-1}(c)$ with $|x'-x|\le C(n,N,m)\delta r_i$. By Lemma \ref{level sets of phi_i is graphical}, $\Phi_i^{-1}(c)$ is a $(C(n,N,m)\delta, r_i)$-graphical submanifold, with graphing plane $x'+\ker \nabla \Phi_i(x')$ in $B_{r_i}(x')$. Note that $d(\ker \nabla \Phi_i(x'),\ker \nabla \Phi_i(x))\le C(n,N,m)\delta$. We have
\begin{align}
\begin{split}
&d_H|_{B_{0.9r_i}(x)}(\Phi^{-1}(c),l_{x,r_i}^\perp)\\
\le\,&d_H(\Phi^{-1}(c),\Phi^{-1}_i(c)) + d_{H}|_{B_{r_i}(x')}(\Phi_i^{-1}(c), \ker \nabla\Phi_i(x') + x')\\ 
&+ d_{H}|_{B_{r_i}(x')}(\ker \nabla\Phi_i(x') + x', \ker \nabla\Phi_i(x) + x) + d_H|_{B_{r_i}(x)}(\ker \nabla\Phi_i(x) + x,l_{x,r_i}^\perp)\\
\le\,& C(n,N,m)\delta r_i.
\end{split}
\end{align}
\epf

The following lemma allows us to pick points increasingly near to $S$ inductively on a given level set of $\Phi$.
\begin{lem}\label{finding a point on S cap level set}

Assume the same conditions as in Corollary \ref{bound of L(Phi_i)}. For $i\ge 1$, $r_i=2^{-mi}$, assume $x_i\in \Phi^{-1}(c)$ with $B_{r_i}(x_i)\subset B_{1.99}(0^n)$, $S\cap B_{0.39r_i}(x_i)\neq\emptyset$. Then there exists $x_{i+1} \in B_{0.5 r_i}(x_i)\cap \Phi^{-1}(c)$ such that $S\cap B_{0.39 r_{i+1}}(x_{i+1})\neq \emptyset$.
\end{lem}
\pf
By Lemma \ref{level sets of Phi are locally almost planes}, 
\begin{equation}   
    d_H|_{B_{0.9r_i}(x_i)}(\Phi^{-1}(c),l_{x_i,r_i}^\perp)\le C(n,N,m)\delta r_i.
\end{equation}
Since $S$ is a $(k,\delta,N)$-splitting Reifenberg set, there exists $A$ with $A=\hat{l}_{x_i,r_i}+A$, such that
\begin{equation}
    d_H|_{B_{r_i}(x_i)}(S,A)\le C(n,N,m)\delta r_i
\end{equation}
Take $y\in S\cap B_{0.39 r_i}(x_i)$, and let $y'$ be the projection of $y$ on $l_{x_i,r_i}^\perp$. Then $d(y',S)\le C(n,N,m)\delta r_i$, $d(y', \Phi^{-1}(c))\le C(n,N,m)\delta r_i$. Thus we may pick $x_{i+1}$ near $y'$ such that $x_{i+1} \in B_{0.5 r_i}(x_i)\cap \Phi^{-1}(c)$ and $S\cap B_{0.39 r_{i+1}}(x_{i+1})\neq \emptyset$.
\epf

Now we can prove Proposition \ref{distance between level sets of Phi intersecting S}.

\pf[Proof of Proposition \ref{distance between level sets of Phi intersecting S}]
It only remains to prove the upper bound. Let $c,d\in B_{1.98}(0^n)$, $|c-d|\le 1$, $x\in \Phi^{-1}(c)\cap S$, $d_0=d(x,\Phi^{-1}(d))$. Then $B_{r_1}(x)\in B_{1.99}(0^n)$ If $d_0\ge 0.3 r_1$, the conclusion holds since 
\begin{equation}
    d_H|_{B_{2}(0^n)}(S,\Rb^k\times \{0^{n-k}\})\le \delta. 
\end{equation}
If not, then there exists an $i\ge 1$ such that $0.3r_i \ge d_0>0.3 r_{i+1}$. Note that $\Phi^{-1}_i(d)\cap B_{0.39r_i}(x)\neq \emptyset$. Let $\hat{l}_{x,r_i}$ be the splitting direction of $S$ in $B_{x,r_i}$. By the same technique as in the proof of Lemma \ref{level sets of Phi are locally almost planes}, there exists some $a\in \hat{l}_{x,r_i}, |a|\le 0.4 r_i$, $A\subset \Rb^n$ with $A=A+\hat{l}_{x,r_i}$, such that
\begin{equation}
    d_H|_{B_{0.9r_i}(x)}(\Phi^{-1}(c),l_{x,r_i}^\perp)\le C(n,N,m)\delta r_i.
\end{equation}
\begin{equation}
    d_H|_{B_{0.9r_i}(x)}(\Phi^{-1}(d),l_{x,r_i}^\perp+a)\le C(n,N,m)\delta r_i.
\end{equation}
\begin{equation}
    d_H|_{B_{r_i}(x)}(S,A)\le C(n,N,m)\delta r_i.
\end{equation}
Thus $|a-d_0|\le C(n,N,m)\delta r_i$. $d(x+a,S)\le C(n,N,m)\delta r_i$, $d(x+a,\Phi^{-1}(d))\le C(n,N,m)\delta r_i$. Thus we may pick $x_j\in \Phi^{-1}(d)$ with 
\begin{equation}
    d(x_j,x+a)\le C(n,N,m)\delta r_i,
\end{equation}
and $r_j \le C(n,N,m)\delta r_i$ such that $S\cap B_{0.39r_j}(x_j)\neq \emptyset$.
Apply Lemma \ref{finding a point on S cap level set} inductively, and we get a sequence of points $\{x_l\}_{l=j}^\infty$ with $x_l\in \Phi^{-1}(d)$, $|x_{l}-x_{l-1}|\le 0.5r_{l-1}$ and $d(x_l,S)\le 0.39r_l$. Let $x_\infty:=\lim_lx_l$. Since $S$ is closed $x_\infty \in \Phi^{-1}(d)\cap S$, and $|x_\infty-x_j|\le r_j\le C(n,N,m)\delta r_j$. Thus 
\begin{equation}
    d(x_\infty,x)\le (1+C(n,N,m)\delta)d_0\le (1+C(n,N,m)\delta)|c-d|^{1-\alpha}.
\end{equation}
The proof is now finished.
\epf

Finally we can prove our main theorem.
\begin{thm}

Given $\alpha>0$, $n$, $N$, assume $\delta<\delta(n,N,\alpha)$ is sufficiently small. Let $S\subset B_2(0^n)$ be a $(k,\delta,N)$-Reifenberg set. Then there exists a $C^{1-\alpha}$-biH\"older map $\iota: B_1(0^k)\to (2^{S},d_H)$: for each $c,d\in B_{1}(0^k)$ with $|c-d|\le 1$, we have $\iota(c),\iota(d)\subset S$, and
\begin{equation}
    (1-C(n,N)\delta)|c-d|^{\frac{1}{1-\alpha}}\le d_H(\iota(c),\iota(d))\le (1+C(n,N)\delta)|c-d|^{1-\alpha},
\end{equation}
such that 
\begin{equation}
S\cap B_1(0^n)\subset \bigsqcup_{c\in B_{1}(0^k)}\iota(c)\cap B_1(0^n).
\end{equation}

\end{thm}

\pf

Take $m$, $\delta$, $\Phi$ as in Corollary \ref{bound of L(Phi_i)}. Note that $m$ is decided by $n$, the constants $C(n,N,m)$ can actually be replaced by $C(n,N)$. By Lemma \ref{bound of cardinality of level set intersect S}, we may define $i:B_{1.98}(0^k)\to 2^{S}_{C(n,N)}, c\mapsto \Phi^{-1}(c)\cap S$. By Proposition \ref{distance between level sets of Phi intersecting S}, since $\frac{1}{1+\alpha}<1-\alpha$, $i(c)$ is disjoint for each $c$ and $i$ is a $C^{1-\alpha}$ biH\"older homeomorphism onto its image.
It is clear that $||\Phi||<(1+C(n,N))\delta$ in $B_1(0^n)$. Taking $\iota(c)=i(\frac{c}{1+C(n,N)\delta})$, then $\iota$ is also biH\"older with
\begin{equation}
    S\cap B_1(0^n)\subset \bigsqcup_{c\in B_{1.01}(0^k)}i(c)\cap B_1(0^n).
\end{equation}
The proof of the main theorem is now finished.
\epf

It remains to prove a control of the number of sheets of $S$ in our construction, i.e., the cardinality $|\Phi^{-1}(c)\cap S|$. In the general case, $\Phi^{-1}(c)\cap S$ can have a positive Hausdorff dimension, as shown in Example \ref{uncountable sheets}. However, if we assume additionally that there are only finite bad scales at each point, then we may bound $|\Phi^{-1}(c)\cap S|$ as well.

First we introduce a technical lemma find distinct scales from a collection of points.

\begin{lem}\label{P technical}

\cite{BGP92}For $l\in \Nb$ there exists $N(l,n)$ such that for any $N(l,n)$ distinct points in $\Rb^n$ one can choose $x_0,\dots, x_l$ among them so that $|x_0x_i|>2|x_0x_{i-1}|$, $2\le i\le l$.

\end{lem}

\pf

We'll prove by induction. It is clear that $N(1,n)=2$ works. Assume $N(l,n)$ has been selected. Let $X$ be the set of all the distinct points. There exists $y,z\in X$ with $d:=d(y,z)=\diam{X}$.

We may pick $C(n)$ balls of radius $d/10$ to cover $X$ since $X\subset B_d(y)$. Let 
\begin{equation}
    N(l+1,n)=C(n)N(l,n)
\end{equation}
Then there must be one ball $B_{d/10}(a)$ such that $|X\cap B_{r/10}(a)|\ge N(n,l)$. Thus we may pick a sequence $\{x_0,\dots,x_l\}\subset X\cap B_{d/10}(a)$ such that $|x_0x_i|>2|x_0x_{i-1}|, \text{ for }2\le i\le l$.

Note that either $d(y,a)\ge r/2$ or $d(z,a)\ge r/2$ holds. Say $d(y,a)\ge r/2$ for example. Then we may take $x_{l+1} = y$. It is clear that $|x_{0}y|\ge 0.5d>2|x_0x_l|$. Thus the construction for $l+1$ is complete.
\epf

Now we can control the cardinality of $|\Phi^{-1}(c)\cap S|$.
\begin{prop}\label{bound of cardinality of level set intersect S}
Assume the same condition as in Corollary \ref{bound of L(Phi_i)}. Assume further that for each $x\in B_{1.99}(0^n)$, there are at most $N'$ scales $s_j=2^{-j}$ such that $A_{x,s_j}$ is not a single plane. Then for each $c\in B_{1.98}(0^k)$,
\begin{equation}
    |\Phi^{-1}(c)\cap S|\le C(n,N').
\end{equation}
\end{prop}
\pf
For $x,y\in \Phi^{-1}(c)\cap S$, $x,y\in B_{1.99}(0^n)$. There is some $i\ge 1$ such that $0.6r_{i+1}\le|x-y|<0.6r_i$. By Lemma \ref{level sets of Phi are locally almost planes}, 
\begin{equation}
    d_H|_{B_{0.9r_i}(x)}(\Phi^{-1}(c),l_{x,r_i}^\perp)\le C(n,N,m)\delta r_i
\end{equation}
Thus $\diam{A_{x,r_i}^\perp}>0.5r_{i+1}=2^{-m-1}r_i$ and $A_{x,r_i}^\perp$ is not a single plane. By Lemma \ref{P technical}, if $|S\cap \Phi^{-1}(c)|>C(n,N',m)$, there exists $x_0,x_1,\dots, x_{N'+1}$ with $|x_0-x_{i+1}|< r_i|x_0-x_i|$ for each $i$. Thus there are at least $N'+1$ scales at $x_0$ such that the splitting set is not a single plane. Thus $|S\cap \Phi^{-1}(c)|\le C(n,N',m)$. Since $m$ can be decided by $n$, the conclusion is true.
\epf
By Proposition \ref{bound of cardinality of level set intersect S}, if we assume a uniform bound $N'$ of the number of bad scales, then the embedding $\iota$ we constructed in the main theorem satisfies $|\iota(c)|\le C(n,N')$ for all $c\in B_1(0^k)$. Thus we have proved the control of the number of sheets of $S$.

\section*{Acknowledgments}

I am grateful to my advisor Professor Aaron Naber at Institute for Advanced Study for his valuable suggestions, encouragement and guidance during my work on this paper. All potential omissions and errors remain my own.

\printbibliography

\end{document}